\documentclass[a4paper,12pt]{scrartcl}

\usepackage[fleqn]{amsmath}
\usepackage{amssymb}
\usepackage{amsthm}
\usepackage{stmaryrd}
\usepackage{bbm}
\usepackage{hyperref}
\usepackage{mathtools}
\usepackage[numbers]{natbib}
\usepackage{xcolor}

\mathtoolsset{showonlyrefs}

\DeclareMathOperator{\Var}{Var}

\DeclareMathOperator{\trace}{tr}

\newcommand{\bE}{\ensuremath{\mathbb{E}}}

\newcommand{\bN}{\ensuremath{\mathbb{N}}}

\newcommand{\bP}{\ensuremath{\mathbb{P}}}
\newcommand{\bQ}{\ensuremath{\mathbb{Q}}}
\newcommand{\bR}{\ensuremath{\mathbb{R}}}

\newcommand{\bZ}{\ensuremath{\mathbb{Z}}}

\newcommand{\ind}{\ensuremath{\mathbbm{1}}}

\newcommand{\cA}{\ensuremath{\mathcal{A}}}

\newcommand{\cC}{\ensuremath{\mathcal{C}}}

\newcommand{\cL}{\ensuremath{\mathcal{L}}}

\newcommand{\cR}{\ensuremath{\mathcal{R}}}
\newcommand{\cS}{\ensuremath{\mathcal{S}}}

\newcommand{\cZ}{\ensuremath{\mathcal{Z}}}

\newcommand{\norm}[1]{\left\Vert \, #1 \, \right\Vert}

\newcommand{\ddx}[1][1]{\ifnum#1=1 \frac{d}{dx} \else \frac{d^{#1}}{dx^{#1}} \fi}
\newcommand{\ddy}[1][1]{\ifnum#1=1 \frac{d}{dy} \else \frac{d^{#1}}{dy^{#1}} \fi}
\newcommand{\ddt}[1][1]{\ifnum#1=1 \frac{d}{dt} \else \frac{d^{#1}}{dt^{#1}} \fi}

\newcommand{\super}[1]{{(#1)}}

\newcommand{\opi}{\overline{\pi}}
\newcommand{\bfy}{\mathbf{y}}

\newcommand{\bZd}{\mathbb{Z}^d}

\DeclareMathOperator{\Mult}{Mult}
\DeclareMathOperator{\Poi}{Poi}

\usepackage{todonotes} 

\theoremstyle{plain}
\newtheorem{thm}{Theorem}[section]  
\newtheorem{prop}[thm]{Proposition}
\newtheorem{cor}[thm]{Corollary}
\newtheorem{lem}[thm]{Lemma}

\theoremstyle{definition}

\theoremstyle{remark}
\newtheorem{rem}{Remark}[section]

\numberwithin{equation}{section}

\title{Random walks on random walks:\\ non-perturbative results in high dimensions}

\author{ 
Stein Andreas Bethuelsen
 \footnote{University of Bergen, Department of Mathematics, Bergen, Norway 
 \newline
 Email: stein.bethuelsen@uib.de}  
 \quad
 Florian V\"ollering   
 \footnote{University of Leipzig, Germany
  \newline
 Email: florian.voellering@math.uni-leipzig.de}  
}

\begin{document}

\maketitle

\abstract{Consider the dynamic environment governed by a Poissonian field  of independent particles evolving as  simple random walks  on $\mathbb{Z}^d$. 
The \emph{random walk on random walks} model refers to a particular stochastic process on $\mathbb{Z}^d$ whose evolution at time $t$ depends on the number of such particles at its location. 
We derive classical limit theorems for this instrumental model of a random walk in a dynamic random environment, applicable in sufficiently high dimensions. 
More precisely, for $d \geq 5$, we prove a strong law of large numbers and large deviation estimates. 
Further, for $d\geq 9$, we obtain a functional central limit theorem under the annealed law. 
These results are non-perturbative in the sense that they hold for any positive density of the Poissonian field. 
Under the aforementioned assumptions on the dimension they therefore improve on previous work on the model.  
Moreover, they stand in contrast to the anomalous behaviour predicted in low dimensions.}


\tableofcontents


\section{Introduction}


Random walks evolving in a random environment constitute  fundamental models in the study of random motions in random media.
First introduced and analysed on $\mathbb{Z}$ in an i.i.d.\ environment by \citet{SolomonRWRE1975}, this model has since been extensively studied in a wealth of natural settings and shown to capture phenomena that fundamentally distinguishes its behaviour from that of a standard simple random walk. Particularly, it may be sub-ballistic, and it may have sub- or super-diffusive scaling limits \citet{KestenKozlovSpitzerRWRE1975,SinaiRWRE1982}.

In recent years, driven by applications in physics and biology, models of a random walk evolving in an itself dynamically evolving environment have gained much interest. For such models one typically distinguishes between cases where the dynamics is in a certain sense either \emph{fast mixing} or \emph{slow mixing}. 

In the fast mixing setting, the environment mixes at a sufficiently fast rate for the random walk to have  increments that are close to independent on large time scales. Consequently, its asymptotic behaviour resembles that of a simple random walk. Recently, in \citet{BethuelsenVolleringRWDRE2022}, we have made fundamental progress for models in this fast mixing setting. We refer to that paper for a more thorough discussion of the fast mixing setting and for further references.

  Conservative particle systems, e.g.\ the exclusion process or the zero-range process, which have, in a certain sense, slowly decaying space-time correlations, are dynamic environments that constitute the slowly mixing setting to which this paper is dedicated. Such models still serve as a major challenge and no general theory has yet been established that describes the limiting behaviour of random walks on top of such environments.  Moreover, in the low dimensional setting, supported by simulation studies and heuristic arguments, it has been argued that the random walk may be sub- or super-diffusive for such environments, see e.g.\ \citet{AvenaThomannRWDREsimulations2012,Huveneers2018}. Rigorous justifications of such \emph{slow-down phenomena} are however limited. We mention \citet{AvenaHollanderRedigRWDRELDP2009,AvenaJaraVolleringLDP2018} who considered such effects at the level of large deviations, and \citet{HuveneersSimenhausCLT2020,JaraMenezesCLT2020} who focused on the invariance principle. 

In the last decade there has been several advances in the slowly mixing setting  within certain perturbative regimes, e.g.\ \cite{RWRWhighdensity2019,RWRWlowdensity2020,HilarioHollanderSidoraviciusSantosTeixeiraRWRW2014,HilarioKiousTeixeira2020,HollanderKestenSidoraviciusRWHDRE2014,HuveneersSimenhausRWSEP2014,SalviSimenhaus2018}, where  the asymptotically behaviour is again shown to be qualitatively like a simple random walk. In these studies it is argued that, within the particular perturbative regime,  the random walk  has a positive linear speed and thus runs away from the slow dynamics of the environment. Consequently, its increments are again close to independent on large time scales, similar to the fast mixing setting.

The main purpose of this current article is to present new results for the so-called random walk on random walks (RWRW) model, one of the caricature models within the slowly mixing setting. In this model,   first studied in \citet{HollanderKestenSidoraviciusRWHDRE2014} and subsequently  in a series of papers \citet{HilarioHollanderSidoraviciusSantosTeixeiraRWRW2014,RWRWhighdensity2019,RWRWlowdensity2020,HilarioKiousTeixeira2020}, the environment is given by a Poissonian field of independent  particles  with density $\lambda>0$ evolving as standard random walks on $\bZ^d$. For instance, \cite{RWRWhighdensity2019} showed that, depending on the jump kernels of the random walk evolving on top of this environment and assuming $\lambda$ sufficiently large, the model has a diffusive scaling under the so-called annealed law. 
For the model on $\bZ$ the same was concluded in  \cite{RWRWlowdensity2020}, assuming $\lambda$ sufficiently small. Moreover, restricting the random walk to perform nearest neighbour jumps, still on $\bZ$, both these statements were obtained in  \cite{HilarioKiousTeixeira2020}, where they additionally proved the strong law of large numbers for all  values of $\lambda$ except possibly two. This latter result was very recently extended to hold for  all  values of $\lambda$ except possibly one, see \cite{KiousRoodriguez2024}.

We conclude in Theorem \ref{thm:main} the strong law of large numbers and large deviation estimates on the displacement of the random walk (on random walks) when $d\geq 5$. Moreover, in Theorem \ref{thm:main2}, we obtain the scaling to a Brownian motion when $d\geq 9$.  In fact, even for $d\geq 5$, we obtain a weak form of a functional central limit theorem. As alluded to in the abstract, these results extend the results of \cite{RWRWhighdensity2019} to all parameter values and finite range jump kernels of the random walk under the aforementioned assumptions on the dimension. 
Moreover,  the large deviation bounds show that linear deviations from the speed are exponentially costly and hence supersede those in \cite{RWRWhighdensity2019}, who derived stretched-exponential  bounds. 

 Our work also contrasts the previous mentioned works  in methodology. Therein the  assumption on the parameter $\lambda$ was a pivotal ingredient to initiate a coarse graining  (or renormalization) procedure  to ensure a minimal linear bound on  the drift. 
  This drift was in turn applied  in order to control certain regeneration times at which the random walk forgets its past observations. 
Unlike these approaches, for the proof of our main results, we first establish uniform bounds on the mixing behaviour for the so-called local environment process, which in turn enable us to apply the general theory that we developed in \cite{BethuelsenVolleringRWDRE2022}.  
To obtain the mixing bounds, we rely on a novel domination result, see Theorem \ref{thm:domination2}, 
combined with coupling arguments and heat kernel estimates for standard random walks. 
 

\subsubsection*{Outline of the paper}

In the following section we first provide a precise definition of the RWRW model and then state our main results on its asymptotic behaviour. 
In Section \ref{ssec:RWRWoutline} we outline the proofs of the main theorems, and the remaining sections are devoted to the details of these proofs.



\section{The model  and the main results}


The RWRW model can be seen as a particular random walk in a dynamic random environment. 
For constructing the environment, independently for each spatial coordinate $x \in \bZd$, let $N_{x,0} \sim Poi(\lambda)$. The allocated number $N_{x,0}$ corresponds to particles at the spatial location $x$ at time $t=0$. These particle then evolve independently as discrete-time standard random walks, which we assume are truly d-dimensional and of finite range.  The aforementioned \emph{Poissonian field of random walks} is the process constituting the evolution of all these particles.

More precisely, as the dynamic random environment, we consider the following process.  For $x \in \bZ^d$ for which $N_x>0$, letting $Y_{t}^{x,\super{i}}, i=1,\dots, N_{x,0}$, denote the position at time $t\geq 0$ of a particle started at $x\in \bZ^d$,  we consider the process $(\omega_t)$ on $\Omega=\{0,1\}^{\bZd}$, $d\geq 1$, where 
\begin{equation}\label{eq:environment}
\omega_t(x) = \left\{\begin{array}{cc}1 & \text{ if }N_{x,t} \geq 1;  \\0 & \text{otherwise}.\end{array}\right. \quad N_{x,t} \coloneqq \sum_{z\in \bZd }\ind_{N_{z,0}>0} \sum_{i=1}^{N_{z,0}} \ind_{Y_{t}^{z,\super{i}}=x}.
\end{equation}
 That is, we only trace whether there  is at least one particle present at a given space-time coordinate or not. We denote by $\bP$ the law of $(\omega_t)_{t\geq 0}$ on $(\Omega^{\bN_0},\mathcal{F})$ where $\mathcal{F}$ is the standard product $\sigma$-algebra generated by the cylinder events. 

 In addition to the environment $(\omega_t)$, for each $i\in \{0,1\}$, we assume given a certain prescribed probability distribution on $\bZd$, denoted by $\alpha(i,\cdot)$.  
 We assume that its range, denoted by $\cR_i := \{ z \in \bZd \colon \alpha(i,z)>0 \}$, is finite. 
 Then, given $\omega=(\omega_t)$, and denoting by $o\in \mathbb{Z}^{d}$ the origin, the evolution of $(X_t)$ is such that $P_{\omega}(X_0 = o) =1$ and, for $t\geq 0$,
\begin{align}
P_{\omega}(X_{t+1} = y+z \mid X_t = y) = \alpha(i, z),\quad  \text{ if }  \omega_t(y) = i, \quad i=0,1.
\end{align}
Thus, the position of $X_{t+1}$ depends on the  environment at its location in the previous time step, i.e.\ whether there is a particle at location $X_t$ or not. The process  $(X_t)$  is the  \emph{random walk on random walks}.  

\begin{rem}
There are slight variations appearing in the literature of how the RWRW model is defined. For instance, \cite{HollanderKestenSidoraviciusRWHDRE2014} studies an adaptation where both the environment and $(X_t)$ evolve in continuous-time.  In \cite{RWRWhighdensity2019} a more general model where the jump kernel of the random walk $(X_t)$ may depend on the exact number of particle at its location is considered. The model introduced above is the same as the one studied in \cite{HilarioHollanderSidoraviciusSantosTeixeiraRWRW2014,RWRWlowdensity2020}, albeit these papers restrict to the model on $\bZ$ and with nearest neighbour jumps. 
\end{rem}

The above law $P_{\omega}$  on $(\bZ^d )^{\bZ_{\geq0}}$ of $(X_t)$,  where we have conditioned on the entire environment, is called the \emph{quenched} law, and its corresponding $\sigma$-algebra is denoted by $\mathcal{G}$. 
We write $P_{\bP}$ for the joint law of $(\omega,X)$ on $\left(  \Omega^{\bN_0} \times (\bZ^d )^{\bN_{0}}\right)$, that is,
\begin{equation}
P_{\bP}(B \times A) = \int_{B}P_{\omega}(A) d\bP(\omega), \quad B \in \mathcal{F}, A \in \mathcal{G}.
\end{equation}
 The marginal law of $P_{\bP}$ on $(\bZ^d)^{\bN_{0}}$ is the \emph{annealed} law of $(X_t)$.


\begin{thm}\label{thm:main}
Consider the RWRW model  on $\bZd$ with $\lambda \in [0,\infty)$ and  $d\geq 5$. 

\begin{itemize}
\item[a)] The Strong law of large numbers (SLLN):

There exists a  $v \in \bR^d$ such that
\begin{equation}\label{eq:SLLN} \lim_{t \rightarrow \infty} \frac{X_t}{t} = v \quad P_{\bP}\text{-almost surely and in }L_1
\end{equation}

\item[b)]  Large deviation bounds (LDBs):

For every $\epsilon >0$ there exists constants $C,c>0$ such that 
\begin{equation}
P_{\bP} \left( \left| \frac{X_t}{t} - v \right| > \epsilon \right) \leq Ce^{-ct}, \quad \text{ for all } t \in \bN,
\end{equation}
with $v\in \bR^d$ as in \eqref{eq:SLLN}.
\end{itemize}
\end{thm}

As already mentioned, Theorem \ref{thm:main} should particularly be compared with the results of  \cite{RWRWhighdensity2019}. 
For the RWRW model on $\bZ^d$, $d\geq 1$, assuming a certain drift assumption when the random walk observes sufficiently many particles, they concluded in  \cite[Theorem 1.1i)]{RWRWhighdensity2019} a SLLN as in Theorem \ref{thm:main}a) for all $\lambda$ sufficiently large. Moreover, as stated in \cite[Theorem 1.1iii)]{RWRWhighdensity2019}, they obtained LDBs as in  Theorem \ref{thm:main}b), however, with a stretched-exponential decay. We emphasise that Theorem \ref{thm:main} contains no restrictions on $\lambda$ nor the drift of the random walk and are, in contrast to the results obtained in \cite{RWRWhighdensity2019}, in this sense non-perturbative. Moreover, we obtain LDBs with an exponential decay that improve on the bounds of \cite[Theorem 1.1iii)]{RWRWhighdensity2019} when the model assumptions overlap.

 To prove that $(X_t)$ has a diffusive scaling limit we in addition require an assumption on the dimensionality of $(X_t)$. Particularly, we say that it is \emph{truly d-dimensional} if the laws $\alpha(i,\cdot)$, $i=0$ or $1$, are non-singular, in the sense that their corresponding covariance $d\times d$-matrices are invertible. 

\begin{thm}\label{thm:main2}
Consider the RWRW model  on $\bZd$ with $\lambda \in [0,\infty)$. 
\begin{itemize}
\item[a)] (weak) annealed Functional Central Limit Theorem (waFCLT):\\
If $d\geq 5$ and there is an $\alpha \in \bR^d$ such that the variance of $\alpha \cdot X_n$ diverges, then there is a slowly varying function  $h\colon [0,\infty) \rightarrow \bR$ such that
\begin{equation}
\left( \frac{\alpha \cdot X_{\lfloor nt\rfloor}- \alpha \cdot v\lfloor nt\rfloor}{\sqrt{nh(n)}} \right)_{t \in [0,1]} \overset{P_{\bP}}{\implies} (B_t)_{t \in [0,1]}
\end{equation}
where $\overset{P_{\bP}}{\implies}$ denotes convergence in distribution under the annealed law $P_{\bP}$ as $n\rightarrow \infty$ with respect to the Skorohod topology and  $(B_t)$  is a standard Brownian motion on $\bR$.\\  
If, moreover, $d\geq 9$, then there is a $\sigma \in (0,\infty)$ such that $h(n)=\sigma$. 
\item[b)] annealed Functional Central Limit Theorem (aFCLT):\\
If $d\geq 9$ and $(X_t)$ is truly d-dimensional, then 
\begin{equation}
\left( \frac{ X_{\lfloor nt\rfloor}- v\lfloor nt\rfloor}{\sqrt{n}} \right)_{t \in [0,1]} \overset{P_{\bP}}{\implies} (B_t^{\Sigma})_{t \in [0,1]}
\end{equation}
where $(B_t^{\Sigma})$ is a Brownian motion on $\bR^d$ with covariance matrix $\Sigma>0$.
\end{itemize}
\end{thm}

Theorem \ref{thm:main2} should particularly be compared with \cite[Theorem 1.1ii)]{RWRWhighdensity2019}, where the aFCLT was obtained for the RWRW model on $\bZ^d$, $d\geq 1$, again assuming $\lambda$ sufficiently large and a drift assumption. Again our results are non-perturbative as we pose no such assumptions.

The proofs of Theorem \ref{thm:main} and Theorem \ref{thm:main2} are presented in the next section. As detailed in Subsection \ref{ssec:RWRWoutline}, they both build on the general approach of  \cite{BethuelsenVolleringRWDRE2022}, where the study of the asymptotic properties of $(X_t)$ are reduced to analysing a certain uniform mixing quantity of the so-called local environment process. For Theorem \ref{thm:main} and Theorem \ref{thm:main2}a) we only require that this mixing quantity is fulfilled, whereas for Theorem \ref{thm:main2}b) we require a certain quantitative bound on the  rate of decay. The appearance of the slowly varying function $h(\cdot)$ in Theorem \ref{thm:main2}a) stems from this approach, which rely on  classical limit theorems for $\phi$-mixing processes, and is presumably an artefact of the proof method. 

To obtain these uniform mixing bounds, we rely heavily on a decomposition of the underlying environment, presented in Subsection \ref{sec:SDfRW}, which enable us to control, in a uniform sense, the influence on the long term future behaviour of $(X_t)$ of the particles that random walk observes.  
For applying this, we require a sufficient decay of the heat kernel estimates for the simple random walks $(Y_t^x)$ constituting the dynamic environment, implying that they are sufficiently transient (e.g.\ that two independent simple random walks meet only a finitely many times when $d\geq 5$).  In fact, the basic assumptions we make for the evolution of  $(Y_t^x)$  is that they satisfy the properties  \eqref{lem:SRWestimates1} and  \eqref{lem:SRWestimates2} of Lemma \ref{lem:SRWestimates} below, which we believe are valid under fairly general assumptions on the jump transitions. 
  
  On the more technical side, we in addition need to assume that $(X_t)$ is finite range and, for Theorem \ref{thm:main2}, that $(X_t)$ is truly d-dimensional. Particularly, under these assumption we conclude that $(X_t)$ has a diverging variance when $d\geq 3$, see Proposition \ref{prop:varInfinite}, and further that $(X_t)$ is, in a certain sense forgetful (or elliptic), as we prove in Proposition \ref{cor:ellipticity}.  
  There seems to be no essential reason for restricting $(X_t)$ to only observing the environment at its location nor to allow for different jump kernels depending on the number of particles, however, these choices are convenient to our analysis at several stages. 

        Our restriction that $(X_t)$ evolves in discrete-time allow us to apply the methods from \cite{BethuelsenVolleringRWDRE2016,BethuelsenVolleringRWDRE2022}. This approach can also be extended to a dynamic environment  evolving in continuous-time following the same approach as outlined in \cite[Section 2]{BethuelsenVolleringRWDRE2016}. It is an interesting problem to extend our approach  to a setting where also the random walk $(X_t)$ has a  continuous-time evolution.       
      
       Presumably, a more refined analysis based on the methods in this paper can work in dimensions $d\geq3$. In fact, we are quite confident that Theorem \ref{thm:main} and Theorem \ref{thm:main2}a) can be extended to hold for $d\geq4$, and Theorem \ref{thm:main2}b) for $d\geq 5$. We will pursue this in a forthcoming follow-up article.  
        However,  \citet{CoxGriffeathPRW1984} have shown that,  when $d\leq2$, the large deviation tails of the occupation time at the origin of $(\omega_t)$ decay slower than exponential. This contrasts the LDBs of Theorem \ref{thm:main}b)
  and indicates that new methods are needed to treat the lower-dimensional cases.  To this end, we should also note that our large deviation bounds contrast those of \cite{AvenaHollanderRedigRWDRELDP2009}, who literarily proved slow-down effects at this level for a random walk model on the exclusion process on $\bZ$, but whose proof may be transferred to the RWRW model without much ado.


\section{Outline of the proofs of the main results} \label{ssec:RWRWoutline}

Following \cite{BethuelsenVolleringRWDRE2022}, the general strategy for the proofs of Theorem  \ref{thm:main} and  Theorem  \ref{thm:main2}  is to control certain mixing properties for the \emph{local environment process} 
$(\xi_t)= (\xi_t^1,\xi_t^2)$ on $\Xi:=\{0,1\} \times \cR$, where $\cR:=\cR_0\cup \cR_1$, and which is  given by
\begin{align}
&\xi_t^{(1)} = \omega_t(X_t),
 &\xi_t^{(2)} = X_{t+1}-X_{t}.
\end{align}
Thus, $\xi_t^{(1)}$ is the environment that $X_t$ needs to observe in order to determine the next jump, whereas $\xi_t^{(2)}$ is simply the jump taken. Moreover, note that
$X_t = \sum_{n=0}^{t-1} \xi_n^{(2)}$.

We next present a mixing quantity for the process $(\xi_t)$ introduced in  \cite{BethuelsenVolleringRWDRE2022}. For this, denote by $\Gamma_{-\infty}$  the set of all possible (i.e.\ deterministic) backwards random walk trajectories from $(o,0)$ of finite length that $(X_t)$ in principle can attain, i.e.\
 \begin{align}
  \label{eq Gamma k} &\Gamma_{-\infty} = \cup_{k \in \bN} \left\{(\gamma_{-k},\dots,o) \colon \gamma_i \in \bZ^d,  \gamma_i-\gamma_{i-1}\in \cR , -k\leq i<0\right\}.
 \end{align}
 Given $\gamma \in \Gamma_{-\infty}$ and $\sigma \in \{0,1\}^{|\gamma|}$, let
\begin{equation}\label{eq Akm}
A(\gamma, \sigma) \coloneqq \bigcap_{i=-{|\gamma|}}^{-1} \left\{ \theta_{\gamma_i,i}\omega= \sigma_{-i} \text{ on } \Delta \right\}
\end{equation}
denote a potential observation of $(\omega_t)$  along the path $\gamma$ and set
\begin{align}
\mathcal{A}(\gamma) \coloneqq \left\{ A(\gamma, \sigma) 
\colon \sigma \in \{0,1\}^{|\gamma|} \text{ and } \bP(A(\gamma,\sigma))>0 \right\}, 
\end{align} 
i.e.\  the set of all observations of $(\omega_t)$ along the path $\gamma$ that are  possible. Finally, let  
\begin{align}\label{eq:LEPmixing}
 \phi(t) \coloneqq  \sup_{\gamma,\gamma'\in \Gamma_{-\infty}} \sup_{\substack{A \in \cA(\gamma)\\ A'\in \cA(\gamma')}} \norm{P_{\bP(\cdot |A)}(\xi_{[t,\infty)}\in \cdot) - P_{\bP(\cdot |A)}(\xi_{[t,\infty)}\in \cdot) }_{TV},
 \end{align}
 where $TV$ stands for the total variation distance. 
 This mixing quantity is an adaptation of the notion of  $\phi$-mixing to non-stationary processes.  
 In \cite{BethuelsenVolleringRWDRE2022} we proved that the decay of $\phi(t)$ towards $0$ transfers to asymptotic properties of the random walk $(X_t)$. For this, we assumed that the random walk was elliptic, which holds for our model as stated next. 

\begin{lem}[Ellipticity]\label{cor:ellipticity}
Consider the RWRW model on $\bZ^d$, $d\geq3$,  with density $\lambda>0$. Then the model satisfies
\begin{equation}\label{cor:eq:ellipticity}
\inf_{\substack{\gamma \in \Gamma_{-\infty}\\ A \in \cA(\gamma)}} \min_{i=0,1} \bP_A(\omega(o) = i )>0. 
\end{equation}
\end{lem}
 Following the convention of \cite{BethuelsenVolleringRWDRE2022}, models satisfying the property \eqref{cor:eq:ellipticity}  are \emph{strongly elliptic}. 
 We postpone the proof of this property to the next section and here first conclude an immediate consequence of this combined with \cite[Proposition 2.1ii)]{BethuelsenVolleringRWDRE2022}.

\begin{prop}
\label{prop:mixing1}
 Consider the RWRW model on $\bZ^d$, $d\geq 1$, with $\lambda \in [0, \infty)$. 
  If $\lim_{t \rightarrow \infty} \phi(t) = 0$, then the SLLN and the LDBs of Theorem \ref{thm:main} holds.
\end{prop}

Similarly, the aFCLTs of Theorem \ref{thm:main2} can be concluded from sufficient decay of $\phi(t)$. However, for this we need an additional properties of $(X_t)$, which we formulate in the following  lemma, and whose proof is presented at the end of this section.

\begin{lem}[Diverging variance] \label{prop:varInfinite}
Assume that $(X_t)$ is truly d-dimensional. Then, whenever $d\geq 3$, it holds that $\Var(\alpha \cdot X_t) \rightarrow \infty$ as $t\rightarrow \infty$ for any $\alpha \in \bR^d \setminus \{o\}$. 
\end{lem}

Note that this property is the reason why in  Theorem \ref{thm:main2} we make the assumption that both jump kernels of $(X_t)$ are truly d-dimensional.

Equipped with Lemma \ref{prop:varInfinite} and  \cite[Proposition 2.1]{BethuelsenVolleringRWDRE2022}, we immediately conclude the following.

\begin{prop}
\label{prop:mixing2}
Consider the RWRW model on $\bZ^d$ with $d\geq 1$ and  $\lambda >0$.
\begin{enumerate}
\item[i)]  If for some $\alpha \in \bR^d \setminus \{o\}$ the variance of $\alpha \cdot X_t$ diverges,  then $(\alpha \cdot X_t)$ satisfies the waFCLT of Theorem \ref{thm:main2}a) whenever $\lim_{t \rightarrow \infty} \phi(t) = 0$. Moreover, the statement holds with $h(n)=\sigma$ for some $\sigma \in (0,\infty)$ whenever $\sum_{t \geq 1} \sqrt{\phi(2^t)}<\infty$.
\item[ii)] If $\sum_{t \geq 1} \sqrt{\phi(2^t)}<\infty$, $(X_t)$ is truly d-dimensional and $d\geq 3$, then $(X_t)$  satisfies the aFCLT of Theorem \ref{thm:main2}b).
\end{enumerate}
\end{prop}

Evidently from Propositions  \ref{prop:mixing1} and   \ref{prop:mixing2}, in order to conclude our main results, it suffices to control the mixing properties of the local environment process.  For this, we prove in Section \ref{sec:uniformMixing} the following key estimate.

\begin{thm}\label{thm:high}
Let $d\geq 5$ and $\lambda>0$. Then there exists $C \in (0,\infty)$ such that $\phi(t) \leq C \log(t)^{-\frac{d}{2}-2}$ for all $t\geq 1$.
\end{thm}

We note here that the proof of Theorem \ref{thm:high} hinges on Theorem \ref{thm:domination2}, presented in the next section and proved in the last section, from which we also derive that the RWRW model is strongly elliptic. Before this, however, we next explain how Theorem \ref{thm:high}  is used to derive our main results.  

\begin{proof}[Proof of Theorems \ref{thm:main}]
 The conclusions follow as a direct application of Theorem \ref{thm:high} to Proposition \ref{prop:mixing1}. Indeed, by Theorem \ref{thm:high}, when $d\geq 5$ we have that $\phi(t)\rightarrow 0$ as $t\rightarrow \infty$. 
\end{proof}

\begin{proof}[Proof of Theorem \ref{thm:main2}]
 Firstly, by Lemma \ref{cor:ellipticity} and Lemma \ref{prop:varInfinite}, $(X_t)$ is strongly elliptic and satisfy $\lim_{t \rightarrow \infty} \Var(X_t)=\infty$. Thus, the conclusions follow as a direct application of Theorem \ref{thm:high} to Proposition \ref{prop:mixing2}. Indeed, by Theorem \ref{thm:high}, when $d\geq 5$ we have that $\phi(t)\rightarrow 0$ as $t\rightarrow \infty$. Hence, the waFCLT  in Theorem \ref{thm:main2}a) is a consequence of Proposition \ref{prop:mixing2}i). Moreover,   when $d\geq 9$, by the bound in Theorem \ref{thm:high}, it follows that   $\sum_{t \geq 1} \sqrt{\phi(2^t)}<\infty$, and thus Theorem \ref{thm:main2}b) follows by Proposition \ref{prop:mixing2}ii).
 \end{proof}
 
\begin{rem}\label{rem:MixingInO}
The mixing property obtained in Theorem  \ref{thm:high} is the key to the proofs of Theorem \ref{thm:main} and Theorem \ref{thm:main2}. To this end, we want to note that we do not think that the decay rate that we obtain for $\phi(t)$ is optimal. In fact, as we  plan to address in a follow-up paper to this,  
we believe  further improvements to the current approach are still possible. 

Motivation for this belief partly stems from the study of the environment along a fixed bi-infinite (random walk) path. See particularly Corollary \ref{prop:mixing_along_d_path} below,  where we prove mixing estimates for the latter object with an improved decay rate when compared to Theorem  \ref{thm:high}. 

For the proof of Theorem  \ref{thm:high}, where we need to consider the random path of $(X_t)$ instead,  with the current methods  we have to compensate for utilising the ellipticity property of the RWRW model to allow $(X_t)$ to  forget the  environment within a certain time-window. This causes a significantly weaker bound on the decay rate than that obtained in Corollary \ref{prop:mixing_along_d_path}.
\end{rem}

We now turn to the proof of Proposition \ref{prop:varInfinite}, which is based on ideas from  \cite{PeresPopovSousiRTRW2013}, see Proposition 1.4 therein and its proof. 

 \begin{proof}[Proof of Proposition \ref{prop:varInfinite}]
We claim that there is some universal constant $C>0$ such that, for any $n \in \bN$ and $\omega = (\omega_t)$, 
\begin{equation}\label{eq:claimVariance}
P_{\omega}( \|X_n -E(X_n)\|)_2>\epsilon)  \geq 1- C \epsilon^d n^{1-d/2},
\end{equation}
where $\|\cdot\|_2$ denotes the Euclidean norm on $\bR^d$. 
From this, by setting e.g.\ $\epsilon=\epsilon(n) = n^{1/3d}$ and applying Chebychevs inequality, it immediately follows that the variance of $\alpha \cdot X_t$ diverges for any $\alpha \in \bR^d \setminus \{o\}$.

To see that the claim \eqref{eq:claimVariance} holds, for $i=0$ or $1$, let $(W_j^\super{i})_{j \geq 1}$ be independent i.i.d.\ sequences of random variables in $\bZ^d$ with common law given by $\alpha(i,\cdot)$, respectively. 
Further, for $R>0$ and $y=(y_1,\dots,y_d) \in \bZ^d$, write $B(y,R) =  \{ (x_1,\dots,x_d) \in \bZ^d \colon \|y-x\|_2 \leq R \}$ and, for a  fixed choice of  $(i_0,i_1) \in \bN_0 \times \bN_0$ with $ i_0+I_1 = n$,
let $W_{[1,i_j]}^\super{j} =\sum_{l=1}^{i_j} W_j^\super{j}$, where $j=0$ or $1$.  Then, recall \citet[Theorem 6.2]{Esseem1968}, which  says that, for any $0\leq r \leq R$, 
\begin{align}
&\sup_{y \in \bR^d} \bP \left( W_{[1,i_0]}^\super{0}   +W_{[1,i_1]}^\super{1}  \in B(y,R) \right)
\\  \leq &C_d  (R/r)^d ( \sup_{u \geq r}  u^{-2} \sum_{k=1}^{i_0} \chi_k^\super{0}(u) + \sum_{k=1}^{i_1} \chi_k^\super{1}(u) )^{-d/2},
\end{align}
where $\chi_k^\super{i}(u)$ is a non-decreasing function of $u>0$ (see \citet[Equation (6.4)]{Esseem1968}  for a precise definition) satisfying $\sup_{u \geq r} u^{-2} \chi_i^\super{i}(u)>0$ whenever $W_j^\super{i}$ is non-singularly distributed. 
Note that, by assumption that the model is truly d-dimensional, both $W_j^\super{0}$ and $W_j^\super{1}$ satisfy the latter condition. 
Therefore, letting $r=1$ and fixing $u\geq 1$ such that $a = \frac{\chi_1^\super{0}(u)}{u^2} \wedge \frac{\chi_1^\super{1}(u)}{u^2}> 0$, it follows that
\begin{equation}
\sup_{y \in \bR^d}  \bP( W_{[1,i_0]}^\super{0}   +W_{[1,i_1]}^\super{1}  \in B(y,R))  \leq C_d  R^d a^{-d/2} n^{-d/2},
 \end{equation}
 for some constant $C_d \in (0,\infty)$.  Now, the number of such $(i_0,i_1)$  is exactly $n+1$. By the union bound it thus follows that,  for some constant $C \in (0,\infty)$ and for any $y \in \bR^d$, 
\begin{equation}
\bP(\ \exists \: (i_0,i_1) \in \bN_0 \times \bN_0\colon i_0+i_1= n \text{ and }W_{[1,i_0]}^\super{0}   +W_{[1,i_1]}^\super{1} \in B(y,R)) \leq C R^d n^{1-d/2}.
\end{equation} 
From this, letting $y=E(X_n)$, we see that \eqref{eq:claimVariance} holds since $X_n$ is equal in distribution to $W_{[1,i_0]}^\super{0}   +W_{[1,i_1]}^\super{1}$ for some (random) $(i_0,i_1) \in \bN_0 \times \bN_0$ satisfying $i_0+i_1=n$. 
 \end{proof}

\section{Properties of the environment}

In this section we first present an essential tool for our proof of Theorem \ref{thm:high}, see Theorem \ref{thm:domination2}. As a first application of this, we there-next present the proof of Lemma \ref{cor:ellipticity} that the RWRW model is strongly elliptic. 

\subsection{A domination results for the environment}\label{sec:SDfRW}

We start by recalling from \cite[Section 2]{RWRWhighdensity2019} a very convenient way of constructing the environment $(\omega_t)$ for the RWRW model that we will make heavily use of in the following. For this, let $S$ be the space of all bi-infinite random walk paths on $\bZ^d$ that the particles may attain. That is, for some fixed $R\in \bN$,
\begin{equation}
S \coloneqq \{ \eta \colon \bZ \rightarrow \bZ^d \colon |\eta(i+1) - \eta(i)|\leq R \: \forall \: i \in \bZ\},
\end{equation}
which we furthermore  endow with the $\sigma$-algebra $\cS$ generated by the canonical coordinates $\eta \rightarrow \eta(i)$, $i\in \bZ$.

As previously introduced, let $(Y^{z,\super{i}})_{z \in \bZd, i \in \bN}$ be a collection of independent random elements of $S$, with each $Y^{z,\super{i}}$ distributed as a double-sided truly d-dimensional random walk on $\bZ^d$ started at $z$ and with range $R$. For $K \subset \bZ^d \times \bZ$, denote by $S_K$ the set of trajectories in $S$ that intersect $K$ and consider the space of point measures
\begin{equation}
H \coloneqq \{ \eta = \sum_i \delta_{\eta_i} \colon \eta_i \in S \text{ and } \eta(S_{\{y\}})<\infty \text{ for every } y \in \bZd \times \bZ \},
\end{equation}
endowed with the $\sigma$-algebra generated by the evaluation maps $\eta \mapsto \eta(S_K)$, for  $K\subset \bZd \times \bZ$.

Now, fix $\lambda>0$ and recall that $N(x,0)$, $x \in \bZ^d$ are i.i.d.\ Poisson$(\lambda)$ random variables. Letting $\eta \in H$ be the random element given by
\begin{equation}
\eta \coloneqq \sum_{z\in \bZd} \sum_{i \leq N(z,0)} \delta_{Y^{z \super{i}}},
\end{equation}
and denoting by $P_{x,0}^{RW}$ the law of a bi-infinite  random walk $Y^{x,\super{i}}$, 
it follows that $\eta$ is a Poisson point process on $H$ with intensity measure given by
\begin{equation}
\lambda \sum_{z \in \bZd} P^{RW}_{(z,0)}.
\end{equation}
In particular, via this construction, we obtain the environment $\omega$ defined in \eqref{eq:environment} simply by setting   $\omega_t(x) = \ind_{N(y,t) \geq 1}$ where
\begin{equation}
N(y,t) \coloneqq \eta( S_{\{y,t\}}), \quad \text{ for } (y,t) \in \bZ^d\times \bZ.
\end{equation}

 We denote by $\bQ$ the law  of $\eta$ on $H$.  Moreover, for $\gamma \in \Gamma_{-\infty}$, we denote  by $\bQ_{\gamma^c}$    the law of $\eta^{\gamma^c}$, that is, the process $\eta$ without any random walk trajectory intersecting $\gamma$.     Then, since the set of trajectories which pass through $\gamma$ and those which do not are disjoint, $\bQ_{\gamma^c}$ is again a Poisson point process.  More precisely, 
 the number of bi-infinite random walk paths in $\bQ_{\gamma^c}$ which intersects $\bZd$ at time $0$ is simply a field of independent Poisson random variables on $\bZd$ with densities $\lambda P_{x,0}^{RW}(Y_{-s} \neq \gamma_{-s}, 1\leq s \leq |\gamma| )$.   Moreover, for positive times, these are forwards in time independent standard random walks, and backwards in time independent standard random walks conditioned to not intersect $\gamma$. 
 
 Similarly, we denote by $\bQ_{\gamma}$  the restriction of $\bQ$ to random walk trajectories whose path go through $\gamma$. From the above construction we thus have that $\bQ =    \bQ_{\gamma} \bigoplus  \bQ_{\gamma^c}$, where we use the notation $\bigoplus$  to emphasise that the samples from   $\bQ_{\gamma}$ and  $\bQ_{\gamma^c}$  are independent.
 
  More generally, for $A \in \cA(\gamma)$ with $\gamma \in \Gamma_{-\infty}$, we denote by $\bQ_A$ the law $\bQ(\cdot \mid A)$. Since the conditioning only involves random walk trajectories which intersects $\gamma$, this can similarly be written as $\bQ_A = \bQ_{\gamma^c} \bigoplus \bQ_{A,\gamma}$. 
Particular, $\bQ_A =\bQ_{\gamma^c}$ if $A=\{ \omega_s(\gamma_s) = 0 \text{ for all } s \in \{- |\gamma|, \dots, -1\} \}$.

The above construction of the underlying environment directly inherits  a natural monotone coupling of $\bQ_A$ and $\bQ_{\gamma^c}$. For this, consider the partial ordering by set-inclusion, i.e.\ $B_1\leq B_2$ if $B_1\subset B_2$, for the Poisson point processes we study. As standard, we say that a measure $\bQ_2$ stochastically dominates $\bQ_1$, written $\bQ_1 \preceq \bQ_2$, if  $\bQ_1$ and $\bQ_2$ can be coupled such that with probability $1$ any realization $(B_1,B_2)$ satisfies $B_1\leq B_2$. By the above construction, the following trivially holds. 

\begin{lem}\label{lem1}
For any $\gamma \in \Gamma_{-\infty}$ and $A \in \cA(\gamma)$ 
we have that 
$\bQ_{\gamma^c}  \preceq \bQ_A$. 
\end{lem}

The next statement, crucial to our analysis of the RWRW model, yields a way to control the law $\bQ_A$ of Lemma \ref{lem1} from above. Informally it says that $\bQ_{A}$ equals (in law) $\bQ_{\gamma^c}$ plus a number of independent particle trajectories intersecting $\gamma$. The exact number and positions of these additional particles depends strongly on $A$ in an apparently highly non-trivial way. However, by enumerating the particles according to their \emph{anchor point}, that is, the most recent point at which they intersect $\gamma$, we can assure that there are at most $n+Poisson(\lambda)$ of them at any location of $\gamma$, for some  $n=n(\lambda) \in \bN$ not depending on any of the other parameters.

\begin{thm}\label{thm:domination2}
Let $\gamma \in \Gamma_{-\infty}$, $A \in \cA(\gamma)$ and $\lambda>0$. Then it holds that, in law, 
\begin{equation}\label{eq:domi22}
\bQ_A= \bQ_{\gamma^c} \bigoplus  \cL (\sum_{i\in \cZ} \delta_{Z^\super{i}}),
\end{equation}
  where $ \cL (\sum_{i\in \cZ} \delta_{Z^\super{i}})$ is the law of $\sum_{i\in \cZ} \delta_{Z^\super{i}}$ and the $Z^\super{i}$'s are random walk paths in $S$ that intersects $\gamma$ having the following properties, depending on an auxiliary random variable $\bfy$:
\begin{enumerate}
\item \label{enum:rw} random walk paths:
\begin{enumerate}
   \item The $Z^\super{i}$'s have the conditional law $P^{RW}(\;\cdot \;|\; D_i)$, where the random event $D_i $  
 only concerns the behaviour of $Z^\super{i}$ in the time interval $[-|\gamma|,-1]$.   
    \item We associate an anchor point $z_i \in \{-|\gamma|,...,-1\}$ to each $Z^\super{i}$.
          This point in particular satisfies $Z^\super{i}_{z_i} = \gamma_{z_i}$, and is a function of $\bfy$.
    \item 
    $D_i$ satisfies $\{Z^\super{i}_{z_i}=\gamma_{z_i}\}\cap\{ Z^\super{i}_s \neq \gamma_s, z_i<s\leq -1 \} \subset D_i \subset \{Z^\super{i}_{z_i}=\gamma_{z_i}\}$ and is a deterministic function of $\bfy$.
\end{enumerate}

\item \label{enum:independence} independence and conditional independence:
\begin{enumerate}
      \item conditioned on $\bfy$ the $Z^\super{i}$'s are independent.
      \item $\eta^{\gamma^c}$ and $(\bfy, \cZ, Z^\super{i}, 1\leq i\leq \cZ)$ are independent.
\end{enumerate}

\item \label{enum:domination} stochastic domination:

\begin{enumerate}
\item There is an $n=n(\lambda) \in \bN$ such that $|\cZ_s| \preceq Poi(\lambda)+n$, where $\cZ_s := \{ i : z_i = s \}$, $s\in \{-|\gamma|,\dots,-1\}$.
\item The vector $(|\cZ_s|)_{s=-|\gamma|,\dots,-1}$ is stochastically dominated by a vector of independent $Poi(\lambda)+n$-distributed random variables.
\end{enumerate}

\item  \label{enum:bounds} random walk bounds when $d\geq 3$:
\begin{enumerate}
      \item there is $\delta>0$ not depending on $\gamma$ nor $A$ such that $P^{RW}(D_i)>\delta$.
\item there exists a constant $C<\infty$, not depending on $\lambda, \gamma$ nor $A$, so that 
\begin{equation}\label{enum:bound} 
\bQ_A(Z^{i}_t = x \;|\; \bfy, z_i) \leq C(t-z_i)^{-\frac d2 }.
\end{equation}
\end{enumerate}
\end{enumerate}
\end{thm}

Theorem \ref{thm:domination2} follows from a more general theorem, applicable to a larger class of conditional Poisson point processes, see Theorem \ref{thm:domination2GeneralOlD}. Details to this and the proof of  Theorem \ref{thm:domination2} are presented in Section \ref{sec:cpp}.


\subsection{Strong ellipticity for the RWRW  model}\label{sec:elliptic}

As a first application of Theorem \ref{thm:domination2},  combined with  standard random walk estimates and Lemma \ref{lem1}, we here present the proof of Lemma \ref{cor:ellipticity}. 
 For this, we first collect a couple of basic estimates for simple random walks, whose proof we postpone to the last subsection. 

\begin{lem}\label{lem:SRWestimates}
Let $(Y_t)_{t \in \bZ}$ be a simple random walk on $\bZd$ with $d\geq 3$.
\begin{enumerate}
\item There exists $\delta>0$ such that 
\begin{equation}\label{lem:SRWestimates1} 
\inf_{\gamma\in \Gamma_{-\infty}} P_{o,0}^{RW} \big( Y_i \neq \gamma_i \text{ for all } i \in [-|\gamma|,-1] \big)>\delta.
\end{equation}
\item There exists $C<\infty$ 
such that, for any $\gamma \in \Gamma_{-\infty}$, $t \leq |\gamma|$ and $s>t$, 
\begin{equation} \label{lem:SRWestimates2}
\sup_{z \in \bZ^d} P_{o,0}^{RW} (Y_{-s} = z \mid Y_i \neq \gamma_i \text{ for all } i \in [-t,-1] ) \leq Cs^{-d/2}. \end{equation}
\end{enumerate}
\end{lem}

To conclude that the RWRW model is strongly elliptic amounts to showing that there is an $\epsilon>0$ such that, for any $A \in \cA(\gamma)$ with $\gamma \in \Gamma_{-\infty}$,
 \begin{equation}\label{eq:ellipicityA)}
  \bQ_A( N(o,0)=0) > \epsilon \: \text{ and } \:  
 \bQ_A( N(0,0)\geq 1) > \epsilon.
 \end{equation}

This is a fairly direct application of Theorem \ref{thm:domination2} that we use in order to prove the following lemma. Here and in the following we abuse notation slightly and also write $\bQ_A$ to denote the natural coupling of $\bQ_A$ and $\bQ_{\gamma^c}$ using Theorem \ref{thm:domination2}, where $\gamma \in \Gamma_{-\infty}$ and $A \in \cA(\gamma)$.

\begin{lem}\label{lem:RWRWellip1}
 There exists $\epsilon_1>0$ so that, for all $A \in \cA(\gamma)$ with $\gamma \in \Gamma_{-\infty}$
\begin{align}\label{eq:domination-no-difference}
\bQ_A(\eta^A(S_{\{(o,0)\}}) = \eta^{\gamma^c}(S_{\{(o,0)\}})) = \bQ_A(\sum_{i\in\cZ} \ind_{Z^\super{i}_0 =o}=0) \geq \epsilon_1.
\end{align}
\end{lem} 

We postpone the proof of Lemma \ref{lem:RWRWellip1} to the end of this subsection.

\begin{lem}\label{lem:RWRWellip2}
There exists $\epsilon_2>0$ so that, for any $\gamma\in \Gamma_{-\infty}$, 
\[ \bQ_{\gamma^c}(N(o,0)=0)\geq \epsilon_2 \text{ and }  \bQ_{\gamma^c}(N(o,0)\geq 1)\geq \epsilon_2.\] 
\end{lem}
\begin{proof}
This holds since, under $\bQ_{\gamma^c}$, the random variable $N(o,0)$ is Poisson-distributed with parameter $\lambda P_{o,0}^{RW}(Y_s \neq \gamma_{-s} \text{ for all }1\leq s\leq |\gamma| )$, which by Lemma \ref{lem:SRWestimates} is bounded away from $0$ when $d\geq3$.
\end{proof}

\begin{proof}[Proof of Proposition \ref{cor:ellipticity}]
Combining Lemma \ref{lem:RWRWellip1} with Lemma \ref{lem:RWRWellip2} and the independence from Theorem \ref{thm:domination2}, Property \ref{enum:independence}, yields that \eqref{eq:ellipicityA)} holds with 
$\epsilon=\epsilon_1\epsilon_2>0$.
\end{proof}

\begin{proof}[Proof of Lemma \ref{lem:RWRWellip1}]
We use the conditional independence from  Theorem \ref{thm:domination2}, Property \ref{enum:independence},  
and the conditional union bound to write 
\begin{align}
&\bQ_A(\sum_{i\in\cZ} \ind_{Z^\super{i}_0 =o}=0)
= \bE_{\bQ_A} \left[ \prod_{i\in\cZ} \bQ_A( Z^\super{i}_0 \neq o\;|\;\bfy)	\right]	\\
&\geq \bE_{\bQ_A} \left[ \!\!\!\! \: \:\prod_{i\in\cZ, Z_i>-R} \!\!\!\! \bQ_A( Z^\super{i}_0 \neq o\;|\;\bfy) \cap \big(1-\!\!\!\!\sum_{i\in\cZ, z_i\leq -R}\!\!\!\! \bQ_A( Z^\super{i}_0 = o\;|\;\bfy) \big) \right]	\label{eq:rw-ellipticity-1}
\end{align}
where $R \in \bN$ is a quantity to be determined later. 
Further, for $z_i>-R$, by the Property \ref{enum:rw} of  Theorem \ref{thm:domination2},
\begin{align}
\bQ_A\big( Z^\super{i}_0 \neq o \;\big|\; \bfy\big) 
&= P^{RW}\big( Y_0 \neq o \;\big|\; D_i(\bfy) \big)	\\
&\geq P^{RW}\big( Y_0 \neq o , D_i(\bfy) \;\big|\; Y_{z_i}=\gamma_{z_i} \big) 	\\
&\geq P^{RW} \big( Y_0 \neq o , Y_s \neq \gamma_s, z_i<s\leq -1 \;\big|\; Y_{z_i}=\gamma_{z_i} \big).
\end{align}
In particular, by Lemma \ref{lem:SRWestimates} there is a constant $\delta>0$ such that
\begin{align}\label{eq:rw-ellipticity-1a}
\prod_{i\in\cZ, Z_i>-R} \!\!\!\! \bQ_A ( Z^\super{i}_0 \neq o\;|\;\bfy)
\geq  \prod_{s=1}^{R-1} \delta^{|\cZ_{-s}|}.
\end{align}

When $z_i \leq -R$, again by Properties \ref{enum:rw}  of  Theorem \ref{thm:domination2},
\begin{align*}
\bQ_A \big( Z^\super{i}_0 = o \;\big|\; \bfy\big) 
&=  P^{RW}\big( Y_0 = o \;\big|\; D_i(\bfy) \big) 		\\
&\leq
\frac{P^{RW}( Y_0 = o |Y_{z_i}=\gamma_{z_i})}{P^{RW}(Y_s \neq \gamma_s, z_i<s\leq -1)}.
\end{align*}
As above, by Lemma \ref{lem:SRWestimates}, the denominator is greater than $\delta>0$ and, moreover, there exists some constant $C<\infty$ so that 
\begin{align}
\bQ_A \big( Z^\super{i}_0 = o \;\big|\; \bfy\big) \leq C |z_i|^{-\frac d2}.
\end{align}
We now choose $R$ large enough so that $C R^{-\frac d2}<1$. 
Then 
\begin{align}
\prod_{i\in\cZ, i\geq R} (1-\bQ_A( Z^\super{i}_0 = o\;|\;\bfy) )
&\geq \prod_{i\in\cZ, i\geq R} (1-C |z_i|^{-\frac d2} )   \\
&= \prod_{s=R}^{|\gamma|} (1-C s^{-\frac d2} )^{|\cZ_{-s}|}.		\label{eq:rw-ellipticity-1b}
\end{align}

Putting \eqref{eq:rw-ellipticity-1a} and \eqref{eq:rw-ellipticity-1b} into \eqref{eq:rw-ellipticity-1} we obtain
\[
\bQ_A(\sum_{i\in\cZ} \ind_{Z^\super{i}_0 =o}=0) \geq \bE_{\bQ_A} \left[ \prod_{s=1}^{R-1} q_{-s}^{|\cZ_{-s}|} \prod_{s=R}^{|\gamma|} (1-C s^{-\frac d2} )^{|\cZ_{-s}|}\right].
\]
Since the product is decreasing in $|\cZ_{-s}|$ we can use the domination property Theorem \ref{thm:domination2}, Property \ref{enum:domination}, by a product of independent $(Poi(\lambda)+n)$-variables, so that we get
\begin{align}\label{eq:elliptic-last}
\bQ_A (\sum_{i\in\cZ} \ind_{Z^\super{i}_0 =o}=0) \geq \prod_{s=1}^{R-1} \bE [q_{-s}^{Poi(\lambda)+n} ]\prod_{s=R}^{|\gamma|} \bE [ (1-C s^{-\frac d2} )^{Poi(\lambda)+n}].
\end{align}
This is clearly positive, but we have to show that it is bounded away from $0$ for arbitrary large $|\gamma|$.
Since the exponential moment of a Poisson random variable is explicit we get
\begin{align*}
\prod_{s=R}^{|\gamma|} \bE [(1-C s^{-\frac d2} )^{Poi(\lambda)+n}]
&= \exp\left( - \sum_{s=R}^{|\gamma|} \left(\lambda Cs^{-\frac d2}-n \log(1-Cs^{-\frac d2}) \right) \right),
\end{align*}
from which, when applied to  \eqref{eq:elliptic-last}, it follows that \eqref{eq:domination-no-difference} holds when $d\geq 3$.
\end{proof}

\subsection{Random walk estimates}

 We end this section with the proof of Lemma \ref{lem:SRWestimates}. 
 
\begin{proof}[Proof of Lemma \ref{lem:SRWestimates}]
These estimates follow by  classical heat kernel estimates. For instance, by  \citet[Proposition 2.4.4]{LawlerLimic2010}, we know that, for some $c\in (0,\infty)$, it holds that
\begin{equation}\label{lem:SRWestimatesHelp1}
P_{o,0}^{RW} (Y_{-t} = y) \leq ct^{-d/2}, \quad \text{ for all } y \in \bZ^d \text{ and } t \in \bN.
\end{equation} 
That the bound in  \eqref{lem:SRWestimates2} holds follows by \eqref{lem:SRWestimatesHelp1} and \eqref{lem:SRWestimates1} and, thus, it is sufficient to show the latter estimate. For this, fix $\gamma \in \Gamma_{-\infty}$ and,  for $n \in \{1,\dots,|\gamma|\}$, let 
\begin{equation}
D_n^{(\gamma)}= \{ y \in \bZ^d \colon  P_{o,0}^{RW}(Y_n=y  \mid Y_{i} \neq \gamma_{i} \text{ for all } i \in [-n,-1]) >0\}.
\end{equation} 
Then, as follows e.g.\ by a basic induction argument, we have that 
\begin{equation}
D_n^{(\gamma)}= \{y \in \bZ^d \colon P_{o,0}^{RW}(Y_{-n}=y) >0 \} \setminus \{\gamma_{-n}\}.
\end{equation}
Therefore, by \eqref{lem:SRWestimatesHelp1}, 
\begin{equation}\label{lem:SRWestimatesHelp2}
\lim_{n\rightarrow \infty} \inf_{\gamma \in \Gamma_{-\infty} \colon|\gamma|\geq n} P_{o,0}^{RW}\big(Y_{-n} \in D_n^{(\gamma)}  \big)  = 1 
\end{equation}
Further, since there are only finitely many paths of length $n$ in $\Gamma_{-\infty}$ and $(Y_n)$ has finite range, we have that, for each $n \in \bN$, 
\begin{equation}
\epsilon_1^{(n)} = \inf_{\gamma \in \Gamma_{-\infty}} \inf_{y \in D_n^{(\gamma)}} P_{o,0}^{RW} \big(  Y_{i} \neq \gamma_{i} \text{ for all } i \in [-n,-1], Y_{-n} =y  \big)>0.
\end{equation}

Moreover, again utilising  \eqref{lem:SRWestimatesHelp1}, by a union bound estimate, we have that,
\begin{equation}\label{lem:SRWestimatesHelp3}
\lim_{n\rightarrow \infty} \inf_{\gamma \in \Gamma_{-\infty} \colon |\gamma| \geq n} P_{o,0}^{RW} ( Y_{i} \neq \gamma_{i} \text{ for all } i \in [-|\gamma|,-n])=1
\end{equation}
Now, for any $n \in \{1,\dots, k\}$ where $k=|\gamma|$, we have  that
\begin{align}
P_{o,0}^{RW} &\big( Y_i \neq \gamma_i \text{ for all } i \in [-|\gamma|,-1]\big) 
\\&=   \sum_{y\in \bZ^d} P_{o,0}^{RW} \big(  Y_{-n} =y,  Y_i \neq \gamma_i \text{ for all } i \in [-|\gamma|,-1]\big)   \big) 
\\& \geq \sum_{y\in D_n^{(\gamma)}}  P_{o,0}^{RW} \big(Y_i \neq \gamma_i \text{ for all } i \in [-n,-1], Y_{-n} =y  \big) 
\\ & \qquad \qquad \cdot P_{o,0}^{RW} \big( Y_i \neq \gamma_i \text{ for all } i \in [-k,-n], Y_{-n} =y  \big).
\\& \geq \epsilon_1^{(n)}  P_{o,0}^{RW} \big( Y_i \neq \gamma_i \text{ for all } i \in [-k,-n], Y_{-n} \in D_n^{(\gamma)}  \big). 
\end{align}
From this we conclude that \eqref{lem:SRWestimates1} holds since, by \eqref{lem:SRWestimatesHelp2} and \eqref{lem:SRWestimatesHelp3},  we have that 
 $\inf_{\gamma \in \Gamma_{-\infty} \colon |\gamma| >n}P_{o,0}^{RW} \big( Y_i \neq \gamma_i \text{ for all } i \in [-k,-n], Y_{-n} \in D  \big) >0$ for some $n \in \bN$ sufficiently large. 
\end{proof}

 \section{Uniform mixing for the local environment process} \label{sec:uniformMixing}

This section is devoted to the proof of  Theorem \ref{thm:high}. For this, fix $\gamma  \in \Gamma_{-\infty}$ and $A \in \cA(\gamma)$ and consider the decomposition of $\bQ_A$ obtained from Theorem \ref{thm:domination2} yielding that $\eta^A = \eta^{\gamma^c}+\sum_{j}\delta_{Z^\super{j}}$. We next detail how $\eta^A$ 
can be further disintegrated by restricting to observations along a one-dimensional (random walk) trajectory. Particularly, denote by 
\begin{equation}
\Gamma_{+\infty} \coloneqq \cup_{k \in \bN} \left\{(o,\dots,\gamma_{k}) \colon \gamma_i \in \bZ^d,  \gamma_i-\gamma_{i-1}\in \cR , 0< i\leq k \right\}.
\end{equation}
the set of all finite  forward-in-time trajectories that $(X_t)$ in principle can attain. 
Then, for  $\gamma' \in \Gamma_{+\infty}$ and $t \in \{1,\dots, |\gamma'|-1\}$, let 
\begin{equation} 
\lambda(\gamma,\gamma',t) \coloneqq \lambda P^{RW}_{(\gamma'_t,t)} \left( Y_i \neq \gamma_i' ,Y_j \neq \gamma_j \text{ for all }  i \in \{1,\dots, t\} \text{ and }j \in \{-|\gamma|,-1\}  \right),
\end{equation}
which corresponds  to the Poisson density of not previously observed particles at time $t$ when $(X_t)$ moved along the path $\gamma'$ in the past. 
Moreover, let 
\begin{equation}
\lambda_* \coloneqq \inf \left\{ \lambda(\gamma,\gamma',0) \colon  \gamma \in \Gamma_{-\infty} \right\}.
\end{equation}
Note that $\lambda_*$ does not depend on $\gamma'$ and that $\lambda_*>0$ whenever $d\geq3$ by Lemma \ref{lem:SRWestimates}. 
This quantity represents the density of particles 
 that may be present at the walkers location at time $t$ independently of $\gamma$ and $A$ and any past observation.

 Now, let $Q_1,Q_2,\dots$ be an i.i.d.\ sequence with $Q_1  \sim \Poi(\lambda_*)$. Further, for each $t\in \bN_0$ independently, let $(R_s^{\super{t}})_{s\geq 0}$ be a Poisson process with rate $1$ and consider 
\begin{align}
R_t(\gamma') := R_{\lambda(\gamma,\gamma',t)-\lambda_*}^{\super{t}}, \quad t\geq 0 \text{ and } \gamma'\in \Gamma_{+\infty}
\end{align}
Note that the $R_t(\gamma')$'s are independent $\Poi(\lambda(\gamma',t)-\lambda_*)$-distributed random variables. These represent the additional Poisson noise for a random walk trajectory $\gamma'$ that, in companion with $(Q_t)$ (which is independent of the particular path) is independent of previous observed particles. 

We next introduce the Poisson noise that depends on the actual observed particles. For this, let 
\begin{equation}\label{eq:ftp}
\tau^\super{j}(\gamma')=\inf\{t\geq 0 \colon Z^\super{j}_t = \gamma'_t\}, \quad \gamma'\in \Gamma_{\infty}
\end{equation} denote the first time the particle $Z^\super{j}$ from Theorem \ref{thm:domination2} intersects $\gamma'$.
Moreover, write
\begin{align}
 &S_t(\gamma') = \sum_{j\in \cZ} \ind_{\tau^j(\gamma')=t}, \quad t\geq 0
 \end{align}
 for the number of these that intersects $\gamma'$ for the first time at time $t$. 
    Letting
 \begin{align}
 &N_t(\gamma') = Q_t+R_t(\gamma')+S_t(\gamma'), \quad t\geq 0,
\end{align}
we then have  that $\sum_{s=0}^{t} N_s(\gamma')$ is the total number of different particles that are seen along $\gamma'$ in the time-window $[0,t]$.

\begin{lem}
For any $\gamma'\in \Gamma_{\infty}$ we have that, in law,  
\begin{align}
	\eta^A_t(\gamma_t') &=  
	Q_t+ R_t(\gamma_t') + 
	S_t(\gamma_{[0,t]}') \\&+ \sum_{s=0}^{t-1} \sum_{k=1}^{N_s(\gamma_s')} \delta_{(\gamma_t',t)}(\gamma_s'+W_t^\super{s,k},t),
\end{align}
where $(W^\super{s,i})_{s,k \geq0}$ are centred independent particles started at time $s$. 
\end{lem}
\begin{proof}
This follows merely by construction, utilising Theorem \ref{thm:domination2}. Note particularly that the latter sum 
represents particles which have been observed in the time-window $[0,t-1]$ which, by the Markov property, are independent of the conditioning and thus evolve as independent simple random walks.
\end{proof}

With the new description of $(\eta^A_t)$ along a possible random walk trajectory, we are ready to construct the local environment process. For this, 
extend the model to include  additionally two independent i.i.d.\ sequences of random variables,  $(W_j^\super{i})_{j \geq 1}$ with $i=0,1$, having marginal  laws given by $\alpha(i,\cdot)$, respectively. 
Then, we define the two random walks $(X_t^\super{\gamma^c})$ and $(X_t^\super{A})$ iteratively as follows. Firstly, we set $X_0^\super{\gamma^c}=X_0^\super{A} = o$ and then, for $t\geq 0$, we let 
\begin{align}
&X_{t+1}^\super{\gamma^c} - X_{t}^\super{\gamma^c} = \ind_{\eta^{\gamma^c}(\{(X_t,t)\})=0} W^\super{0}_{t} + \ind_{\eta^{\gamma^c}(\{(X_t,t)\})>0} W^\super{1}_{t}
\\ &X_{t+1}^\super{A} - X_{t}^\super{A} =  \ind_{\eta^{A}(\{(X_t,t)\})=0} W^\super{0}_{t} + \ind_{\eta^{A}(\{(X_t,t)\})>0} W^\super{1}_{t}
\end{align}
Hence, $(X_t^\super{\gamma^c})$ evolves as the random walk $(X_t)$ in the environment provided by $\bQ_{\gamma^c}$, whereas $(X_t^\super{A})$ evolves as $(X_t)$ in the full environment provided by $\bQ_{A}$. 

Utilising this coupling construction, we next present the key lemma for the proof of  Theorem \ref{thm:high}. 
For this, we write  $Q_{[0,T]}\geq 1$ for the event that $Q_n\geq 1$ for each $n \in [0,T]$. 


\begin{lem}\label{lem:trick1RWRW}
Consider the RWRW on $\bZd$, $d\geq 5$, with density $\lambda>0$. Then there is $C <\infty$ such that, and for any $n\in \bN$,
\begin{equation}\label{eq:RWRWd5}
 \sup_{\substack{A \in \cA(\gamma)\\ \gamma \in \Gamma_{-\infty}}} \norm{ \bQ_{A}(\xi_{(n,\infty)} \in \cdot  \mid Q_{[0,n]} \geq 1)-\bQ(\xi_{(n,\infty)} \in \cdot \mid Q_{[0,n]} \geq 1) }_{TV} \leq C n^{-\frac{d}{2}+2}.
\end{equation}
\end{lem}

Before proving this lemma, we show how it implies Theorem \ref{thm:high}.

 \begin{proof}[Proof of Theorem \ref{thm:high}]
Fix $\gamma_0,\gamma_1 \in \Gamma_{-\infty}$ and $A_0 \in \cA(\gamma_0)$ and $A_1\in \cA(\gamma_1)$.  
Let $T'=C\log(T)$ and  set \begin{equation} \tau^Q=\inf\{t\geq 0: Q_{[t,t+T']}\geq 1\}.\end{equation} Since the $Q's$ are i.i.d.\ $\Poi(\lambda_*)$ we can choose $C$ so that, 
\begin{equation}\bP(\tau^Q<T-T')\geq 1-e^{cT}\end{equation} for some $c>0$ and all $T$ large. Therefore, for any possible cylinder event $B$ of $\xi_{(n,\infty)}$, we have that 
\begin{align}
| \bQ_{A_0}(\xi_{(T,\infty)}& \in B  )-\bQ_{A_1}(\xi_{(T,\infty)} \in B )|
\\ \leq \sum_{s=0}^{T-T'-1} \bP(\tau^Q=s) &|\bQ_{A_0}(\xi_{(T,\infty)} \in B  \;|\;\tau^Q=s)-\bQ_{A_1}(\xi_{(T,\infty)} \in B \; |\;\tau^Q=s)| 
\label{eq:tau-Q-split}\\ +2e^{-cT}.
\end{align}
Moreover, by performing a shift so that $\tau^Q$ is the new time 0,
\begin{align}
&\bQ_{A_0}(\xi_{(T,\infty)} \in B  \;|\;\tau^Q=s)	\\
&=\bE_{A_0}\left[\bQ_{\tilde{A_0}}\left(\xi_{(T-\tau^Q,\infty)} \in B\;\middle|\;Q_{[0,T')}\geq 1\right) \;\middle|\;\tau^Q=s	\right],
\end{align}
where $\tilde{A_0}$ is the extension of $A_0$ also containing the observations of $(\xi_t)$ up to time $s$. 
With the analogous statement for $\bQ_{A_1}$ and by applying Lemma \ref{lem:trick1RWRW} and the triangular inequality, we can thus upper bound \eqref{eq:tau-Q-split} by
\begin{align}
 2C(T')^{-\frac d2 +2} + 2e^{-cT} = C_1(\log(T))^{-\frac d2 + 2} + 2e^{-cT}
\end{align}
for some new constant $C_1<\infty$, which completes the proof. 
\end{proof}
 
 \begin{rem}\label{rem:trick}
 The above proof is similar to the proof of \cite[Theorem 2.2]{BethuelsenVolleringRWDRE2022}, however, the particular condition on the event that $Q_{[0,T]}\geq 1$ in Lemma \ref{lem:trick1RWRW} is different from that of \cite[Lemma 3.1]{BethuelsenVolleringRWDRE2022}. 
 \end{rem}

\begin{proof}[Proof of Lemma \ref{lem:trick1RWRW}]
Firstly, we claim that there is a $C<\infty$ so that, for any forward-path $\gamma' \in \Gamma_{\infty}$, uniformly in $\gamma$ and $A$, $n$, $\gamma'$ and $\bfy$, it holds that 
\begin{equation}
\bQ_A( \exists i\in\cZ: Z^\super{i}_t = \gamma'_t \text{ for some } t>n \mid \bfy )
\leq \sum_{t=n}^{|\gamma'|}\sum_{s=1}^{|\gamma|} |\cZ_{-s}| C(t+s)^{-\frac d2 },
\end{equation} 
Indeed, this follows merely by the union bound and the random walk bounds from Theorem \ref{thm:domination2}, Property \eqref{enum:bounds}. Consequently, since the $\cZ_{-s}'s$ are stochastically dominated by $N + Z_s$ for some large $N \in \bN$ and  independent $Poi(\lambda)$-distributed random variables  $(Z_s)$, we conclude that
\begin{equation}\label{eq:lem:trick1RWRWbound}
\bQ_A( \exists i\in\cZ: Z^\super{i}_t = \gamma'_t \text{ for some } t>n)
\leq  \tilde{C}n^{-\frac d2 +2},
\end{equation}
for some constant $\tilde{C}<\infty$, again not depending on $\gamma', \gamma, A$ nor $n$.

Consider now the above coupling construction of both the environment and the random walks $(X_t^\super{\gamma^c})$ and $(X_t^\super{A})$ for some $\gamma\in \Gamma_{-\infty}$ and $A \in \cA(\gamma)$. 
Then, under the condition that $Q_{[0,n]}\geq 1$, in the time-window $[0,n]$ they necessarily make the same jumps given by the $(W_j^\super{1})_{j=0}^n$ random variables.  Moreover, since both $(Q_i)_{i=0}^n$ and $(W_j^\super{1})_{j=0}^n$ are independent of all other components in the construction, for this did not reveal any information about the environment in in this time-window. Now, for any $m\in \bN$, let
\begin{equation}
\tau_m :=\inf\{t\in \{n+1,\dots,n+m\} \colon \eta^A(\{(X_t^\super{A},t)\})\neq \eta^{\gamma^c}(\{(X_t^\super{\gamma^c},t)\}), Q_t=0 \}
\end{equation} be the  time of the potentially first observed difference in the time interval $[n+1,n+m]$.  
By construction we have that
\begin{align}
&|\bQ_{A}(\xi_{(n,n+m)} \in \cdot  \mid Q_{[0,n]} \geq 1)-\bQ(\xi_{(n,n+m)} \in \cdot \mid Q_{[0,n]} \geq 1)|
 \\ \leq &\bQ_{A} ( \tau_m <\infty | Q_{[0,T]}\geq 1).
 \end{align}
 Moreover, using the independence properties of the construction and the random walk bounds, we have that
 \begin{align}
 &\bQ_{A} ( \tau_m<\infty | Q_{[0,n]}\geq 1)
\\ & = \sum_{\substack{\gamma' \in \Gamma_{+\infty} \\ |\gamma'| =n+m}} \bQ_{A} ( W_{[0,n]}^{\super{1}}= \gamma_{[0,n]}', X_{[0,n+m]}^\super{\gamma^c}= \gamma', \tau_m <\infty \mid Q_{[0,n]}\geq 1)
\\ & =  \sum_{\substack{\gamma' \in \Gamma_{+\infty} \\ |\gamma'| =n+m}}\bQ_{A} ( \tau_m<\infty \mid W_{[0,n]}^{\super{1}}= \gamma_{[0,n]}', X_{[0,n+m]}^\super{\gamma^c}= \gamma', Q_{[0,n]}\geq 1)
\\ &\qquad \qquad \qquad \cdot \bQ_A(W_{[0,n]}^{\super{1}}= \gamma_{[0,n]}', X_{[0,n+m]}^\super{\gamma^c}= \gamma' \mid Q_{[0,n]}\geq 1)
\\ & \leq  \sum_{\substack{\gamma' \in \Gamma_{+\infty} \\ |\gamma'| =n+m}}\bQ_A( \exists i\in\cZ: Z^\super{i}_t = \gamma'_t \text{ for some } t= n, \dots, n+m)  
\\ &\qquad \qquad \qquad \cdot \bQ_A (W_{[0,n]}^{\super{1}}= \gamma_{[0,n]}', X_{[0,n+m]}^\super{\gamma^c}= \gamma' \mid Q_{[0,n]}\geq 1)
\\ & \leq  \sum_{\substack{\gamma' \in \Gamma_{+\infty} \\ |\gamma'| =n+m}} \tilde{C}n^{-\frac d2 +2} \cdot \bQ_A (W_{[0,n]}^{\super{1}}= \gamma_{[0,n]}', X_{[0,n+m]}^\super{\gamma^c}= \gamma' \mid  Q_{[0,n]}\geq 1)
\end{align}
where we in the last inequality applied the bound \eqref{eq:lem:trick1RWRWbound}. Consequently, since the latter sum equals $1$, we have that, for some $\tilde{C}<\infty$,  uniform in all parameters
\begin{align}
|\bQ_{A}(\xi_{(n,n+m)} \in \cdot  \mid Q_{[0,n]} \geq 1)-\bQ(\xi_{(n,n+m)} \in \cdot \mid Q_{[0,n]} \geq 1)| \leq \tilde{C}n^{-\frac d2 +2}.
\end{align}
From this we conclude the statement of  Proposition \ref{lem:trick1RWRW} by use of the triangular inequality.
   \end{proof}
   
We close this subsection with providing the statement alluded to in Remark \ref{rem:MixingInO}. 
For this, consider the fields of random walks on $\bZd$, with density $\lambda>0$ and fix a bi-infinite random walk trajectory, say $\gamma$. Further, let $(Z_i)_{i \in \bZ}$ be the $\{0,1\}$-process which records whether there is a particle or not along this trajectory. That is, we set
\begin{equation}
Z_i= \omega_i(\gamma_{i}), \quad i \in \bZ 
\end{equation}
Then, by reasoning along the lines as in the proof of Lemma \ref{lem:trick1RWRW}, we have that, whenever $d\geq 3$, there is a constant $C <\infty$ such that, for any $n,m,l\in \bN$ and $\sigma \in \{0,1\}^{\bZ}$,
\begin{equation}
|\bQ( Z_{[n,n+m]} \equiv \sigma_{[n,n+m]}  | Z_{[-l,0]} \equiv \sigma_{[-l,0]} ) - \bQ( Z_{[n,n+m]} \equiv \sigma_{[n,n+m]} )| \leq C\sum_{l=n}^{n+m} l^{-\frac{d}{2}+1}.
\end{equation}
In particular, for $d\geq 5$, we obtain that 

\begin{cor}\label{prop:mixing_along_d_path}
Let $d\geq 5$. Then there is a $C <\infty$ such that 
\begin{equation}\label{eq:fRWs5}
\sup_{\sigma \in \{0,1\}^{\bZ}} \sup_{l \in \bN} \norm{ \bQ( Z_{[n,\infty)}  \in \cdot   | Z_{[-l,0]} \equiv \sigma_{[-l,0]} ) - \bQ( Z_{[n,\infty)} \in \cdot ) }_{TV} \leq Cn^{-\frac{d}{2}+2}.
\end{equation}
\end{cor}


\section{Decomposition of conditional point processes
}\label{sec:cpp}

Assume given $\gamma  \in \Gamma_{-\infty}$  and $A \in \cA(\gamma)$ and recall the description of the distribution $\bQ_{A}$ from Subsection \ref{sec:SDfRW} As detailed therein, we can decompose this measure as $\bQ_A = \bQ_{\gamma^c}  \bigoplus \bQ_{A,\gamma}$ where $\bQ_{\gamma^c}$ constitutes all random walk paths that do not intersect $\gamma$ and $\bQ_{A,\gamma}$ those which do.  By the Poissonian construction, only the latter depends on the conditional event $A$.  The purpose of this section is to present the proof of Theorem \ref{thm:domination2} that provides a detailed description of $\bQ_{A}$ by means of a further decomposition of $\bQ_{A,\gamma}$.

In the first subsection, we present a re-formulation that maps the problem at hand to the study of distributions that live on a space of finite and ordered elements. Using this, since we believe the underlying arguments can be of interest well beyond the scope of this paper, we there-next cast the problem in the more general setting  of conditional Poisson point processes. In the then following subsection, we present a first decomposition result that, combined with a basic stochastic domination property for the Poisson distribution,  comes close to conclude Theorem \ref{thm:domination2}, but not quite. A  finer coarsening is needed, as well as  a way to undo the coarsening,  as presented in the remaining subsections, again building on properties of the Poisson distribution. The coarsening allows us to iterate the decomposition and, utilising all the above mentioned ideas,  to conclude the proof of Theorem \ref{thm:domination2}.

\subsection{A first re-formulation}
In order to gain deeper understanding of $\bQ_{A,\gamma}$, we first re-formulate the problem at hand. For this, let $n=|\gamma|$ and set $I = \{0,1\}^n$ which we  equip with the lexicographic order given by  $(x_1,\dots,x_n) < (y_1,\dots, y_n)$ if, for some $k \in \{1,\dots,n-1\}$, $x_i = y_i$ for all $i \leq k$ and $x_k < y_k$. Note that $I$ is in fact then an ordered set.  Further, let $\sigma \in I$ be such that $A = A(\gamma,\sigma) \in \cA(\gamma)$, i.e.\ $\sigma$ is the observation along the path $\gamma$. 
 
 We denote by $O=O(\sigma) \coloneqq \{ i \in \{1,\dots,n\} \colon \sigma_i=1\}$ and $V = V(\sigma) \coloneqq \{ i \in \{1,\dots,n\} \colon \sigma_i=0\}$. Then, letting for each $x \in  I$, 
\begin{equation}
\{X \in x\} := \{ X_{-t} = \gamma_{-t} \: \forall \; t \in O \}\cap \{X_{-t} \not= \gamma_{-t} \: \forall \; t \in V \},
\end{equation}
where $X$ is a bi-infinite random walk path on $\bZd$ representing the evolution of a particle in the environment.
Particularly, with $X$ distributed as the independent particles in \eqref{eq:environment}, the corresponding intensity of random walk particles under $\bQ_{A,\gamma}$ is given by 
\begin{align}
\lambda_x =
\begin{cases}
\lambda \sum_{z\in\bZ^d} P^{RW}(X \in X \;|\; X_0=z),\quad &X(i)=0 \ \forall\; i\in V;\\
0,& \text{otherwise},
\end{cases}
\end{align}
Next, let  $(N_x)_{x \in I}$  be independent Poisson($\lambda_x$) random variables and denote the law of these by $\bP_I$. This law equals that of the projection of $Q_{\gamma}( \cdot | \omega_{i}(\gamma_{-i}) =0, i \in V)$ onto $\gamma$.   To take the full conditioning on $A$ into account, we let 
\begin{align}
C_{i} \coloneqq \{x \in I : x_i = 1, x_j = 0 \: j \in V\}, \text{ for } i \in O
\end{align}
and write $\cC \coloneqq \bigcap_{i\in O} \cC_i$, where $ \cC_i \coloneqq \{\sum_{x\in C_i\cap I}N_x>0\}$. 
Observe that the distribution $\bP_I(\cdot \;|\; \cC)$ stands in one-to-one correspondence with the projection of $Q_{\gamma,A}( \cdot )$ onto $\gamma$. Important to the following, what we have obtained from this re-formulation, is a distribution that takes values in a space of finite and ordered elements.

\subsection{A generalization and a first decomposition}

We now cast the problem of decomposing $\bQ_A$ in a slightly more general setting. 
As above, let $I$ be an ordered index set and let $(N_x)_{x\in I}$  be independent Poisson($\lambda_x$) random variables with $\lambda_x\geq 0$ and $\sup_{x\in I} \lambda_x  <\infty$. We denote the law of these by $\bP_I$, and the law of a subset of these by $\bP_J$, where $J\subset I$ is the index subset. Further, let $C_1, C_2, ..., C_K$ be subsets of $I$, which we call the constraint index sets,  and set $\cC_i=\{\sum_{x\in C_i}N_x>0\}$.  We write $\cC= \bigcap_{i=1}^K \cC_i$ for the collection of all the constraints, and we throughout assume that $\bP_I(\cC)>0$, that is, the constraints can be satisfied. This holds for the particular model of a  Poissonian field of random walks but, in general, this is only automatic if $\lambda_x>0$ for all $x\in I$.

Before presenting our first result, we need to introduce some more notation.  Let $M = I\cap\bigcup_i C_i$ be the set of indices relevant for the constraints, and set $M^c = I\setminus M$.  Further, for $y\in I$, let 
$I_{<y}:=\{x\in I: x<y\}$ and define $I_{> y}$, $I_{\leq y}$ and $I_{\geq y}$ similarly. 
Moreover, let 
 $C_i^y:=C_i\cap I_{<y}$ be the $y$-reduced constraint index sets and, for $i =1,\dots, K$, denote 
the corresponding constraints by $\cC_i^y=\{\sum_{x\in C^y_i}N_x>0\}$. We  write $\cC^y:=\bigcap_{i\in 1,...,K, y \not\in C_i}\cC_i^y$ for the event corresponding to constraints that still are not satisfied in case $N_y>0$ and do depend on indexes $x < y$.

The first lemma gives a decomposition of $\bP_I(\cdot | \cC)$. To state it precisely, let  $Y$ be the random variable that outputs the largest index of a Poisson random variable which helps to satisfy the constraints, i.e.\ 
\begin{equation}
Y =\max\{x \in M : N_x>0\}.
\end{equation}
Moreover, for a cylinder event $B=\bigtimes_{x\in I} B_x$ with $B_x \subset \bN_0$ and $J\subset I$, we write $B_J:= \bigtimes_{x\in J} B_x$.

\begin{lem}\label{lemma:decomposition}
Let $B=\bigtimes_{x\in I} B_x$ be a cylinder event. 
Then, for any constraint set $\cC$, it holds that $\bP_I(B \;|\; \cC)$ equals 
\begin{align}\label{eq:decomposition}
&\prod_{x\in M^c} \bP_x(B_x) \sum_{y \in M} \bP_I(Y=y \;|\; \cC) \prod_{x\in I_{>y}\cap M} \ind_{0\in B_x} \\
&\quad\cdot \bP_y(N_y \in B_y \;|\; N_y \geq 1) \bP_{I_{<y}\cap M}( B_{I_{<y}\cap M} \;|\;\cC^y).
\end{align}
\end{lem}

\begin{proof}
First, we can split of all the Poisson random variables which are not part of the conditioning, which by independence yields that 
\begin{align}
\bP_I(B \;|\; \cC) 
&= \left(\prod_{x\in M^c} \bP_x(B_x)\right) \bP_M(B_M \;|\;\cC).
\end{align}
To control the latter term, we next find the largest index $Y$ for which one of the Poisson random variables on which we condition is positive, which gives us 
\begin{align}
&\bP_M(B_M \;|\;\cC) = \sum_{y \in M} \bP_M(Y=y \;|\; \cC) \bP_M(B_M \;|\; Y=y, \cC)   \\
&= \sum_{y \in M} \bP_M(Y=y \;|\; \cC) \left( \prod_{x\in I_{>y}\cap M} \ind_{0\in B_x} \right) \bP_M(B_{I_{\leq y}\cap M} \;|\; Y=y, \cC).
\end{align}
Here, in the second equality, we have again applied the independence property from the Poisson construction and used that, by definition of $Y$ as the maximum, $N_{>y}=0$ in the sense that $\cap_{y \in I_{>y}} \{N_y =0\}$ holds. Finally, we observe that,
\begin{align}
&\bP_M(B_{I_{\leq y}\cap M} \;|\; Y=y, \cC)
\\ = &\bP_M(B_{I_{\leq y}\cap M} \;|\; N_y\geq 1, N_{>y}=0, \cC)    
\\ = &\bP_{I_{\leq y}\cap M}(B_{I_{\leq y}\cap M} \;|\; N_y\geq 1, \cC^y)
\\ = &\bP_y(N_y\in \cdot \;|\; N_y\geq 1) \bP_{I_{<y}\cap M}(B_{I_{<y}\cap M} \;|\; \cC^y),
\end{align}
where the last equality again follows by independence, and from which we conclude the proof.
\end{proof}

Before presenting our first application of Lemma \ref{lemma:decomposition}, we recall a basic property for the Poisson distribution.

\begin{lem}\label{lem:basicPoisson}
Consider  $Z\sim Poi(\lambda)$ with $\lambda>0$. There is $n=n(\lambda) \in \bN$ such that $\bP(Z\in \cdot \;|\; Z\geq 1) \prec \bP(W_n \in \cdot)$, where $W_n=Z+n$.
\end{lem}
\begin{proof}
Fix $\lambda>0$. It is sufficient to show that, for any $k\geq 0$, it holds that 
\begin{equation}
\bP(Z \geq k \;|\; Z\geq 1) \leq \bP(W_n \geq k),
\end{equation}
for some $n \in \bN$. That this is true for all $k \leq n$ is immediate since then  $ \bP(W_n \geq k)=1$. For $k>n$ we have that 
\begin{align}
 &\bP( Z \geq k \mid Z\geq 1) = \frac{ \bP( Z \geq k)}{ \bP( Z \geq 1)}=   \sum_{l \geq k}^{\infty}e^{-\lambda} \frac{\lambda^l}{l!} (1-e^{-\lambda})^{-1};
\\ &\bP( W_n\geq k) =  \bP( Z \geq k-n)  = \sum_{l \geq k}^{\infty}e^{-\lambda} \frac{\lambda^l}{l!} \frac{l!/(l-n)!}{\lambda^n}
\end{align}
Now, note that the term $\frac{l!/(l-n)!}{\lambda^n}$ in the derivation of $\bP( W_n\geq k)$ always is larger than $\frac{(n+1)!}{\lambda^n}$ since we assume $k>n$. 
The claim thus follows since  $1-e^{-\lambda} >  \frac{\lambda^n}{(n+1)!}$ for all $n$ large.
\end{proof}

Utilising the decomposition in  \eqref{eq:decomposition} via an iterative procedure, and using Lemma \ref{lem:basicPoisson}, we obtain the following domination statement. 

\begin{lem}\label{lem:firstDomiGeneral}
Assume that there is a $\lambda<\infty$ such that $\lambda_x \leq \lambda$ for all $x\in I$. Then we have that $\bP_I( \cdot \;|\; \cC)  \preceq (N_x)_{x\in I}+\sum_{i=1}^\kappa \sum_{j=1}^n \delta_{Y_i^{\super{j}}}$, where 
\begin{enumerate}
\item $\kappa$  is a random number in $\{1,\dots, K\}$
\item $n=n(\lambda)$ is as in Lemma \ref{lem:basicPoisson}
\item The $Y_i^{\super{j}}$'s are random points in $I$.
\item For each $i \in 1,\dots, \kappa$, the $Y_i^{\super{j}}$'s, $j=1\dots,n$, are identical.
\item For $i=1,\dots, \kappa$,  $Y_1^{\super{1}}< \dots < Y_\kappa^{\super{1}}$.
\item Both $\kappa$ and the $(Y_i^{\super{j}})'s$ are  independent of $(N_x)_{x\in I}$.
\end{enumerate}
\end{lem}

\begin{proof}
By Lemma \ref{lemma:decomposition} there exists $Y_1$ so that $\bP((N_x)_{x \in I} \in \cdot \;|\; \cC)$ is stochastically dominated by 
\begin{align}
&\prod_{x\in M^c} \bP(N_x \in \cdot) \sum_{y_1 \in M} \bP(Y_1=y_1 \;|\; \cC) \prod_{x \in I_{>y_1} \cap M} \bP(N_x\in \cdot) \\
&\quad\cdot \bP(N_{y_1} + n \in \cdot ) \bP((N_{x})_{x\in I_{<y_1}\cap M} \in \cdot \;|\;\cC^{y_1}),
\end{align}
where we used that $\delta_0$ is trivially dominated by a Poisson random variable, and also that $\bP(N_{y_1} \in \cdot \;|\; N_{y_1}\geq 1) \prec \bP(N_{y_1}+n \in \cdot)$, where $n=n(\lambda)$ is provided by Lemma \ref{lem:basicPoisson}.  In particular, the distribution of $N_{y_1}$ is independent of $Y_1$.

There are now two possible cases. First, it could be that the conditioning $\cC^{y_1}$ is empty, in which case $\bP((N_{x})_{x\in I_{<y_1}\cap M} \in \cdot \;|\;\cC^{y_1})$ is simply Poisson, and the points $Y_1^{\super{1}}=\dots= Y_1^{\super{n}}$ are the additional points needed to obtain stochastic domination. Second, if $\cC^{y_1}$ is not trivial, we repeat the argument for $\bP((N_{x})_{x\in I_{<y_1}\cap M} \in \cdot \;|\;\cC^{y_1})$, which will yield a second point $Y_2<Y_1$. Since each point $Y_i$ decreases the number of still relevant constraints in $\cC$ by at least $1$, after at most $K$ iterations we have fully satisfied $\cC$. 

Calling $\kappa$ the number of iterations we observe that the dominating vector is given by a product of independent Poisson($\lambda_i$) random variables plus $\kappa$ many extra points, that are independent of the Poisson random variables.
\end{proof}

In the context of the field of random walks, Lemma \ref{lem:firstDomiGeneral} comes close to that of Theorem \ref{thm:domination2}. Particularly, it implies that $\bQ_{A,\gamma}$ is stochastically dominated by the field $\bQ_{\gamma}$ plus an independent set of at most $|O| n$ additional random walk particles, provided by the $Y_i^{\super{j}}$ points, so that each of the constraints are fulfilled (at least once). What we still need to control is that these additional particles do not accumulate at certain space-time locations. This will be obtained by a more thorough description of the decomposition in Lemma  \ref{lemma:decomposition}, and the iteration thereof, as given in the next subsection.

\subsection{Coarsening, and undoing coarsening}

In the remainder of this note, we restrict to the case where $I=\{0,1\}^n$ and, for some fixed $O \subset \{1,\dots,n\}$, assume that the constraints are given by  $C_i = \{x \in I \colon x(i) = 1\}$, $i \in O$. 
 Moreover, to prepare for an  iteration of Lemma \ref{lemma:decomposition}, we add a superscript $^\super1$ to the notions introduced in the previous subsection, e.g.\ we write $I^\super1= I$, $C^\super1_i=C_i$, $O^\super1=O$, $M^\super1=M$. 
 Further, we write  $Y^{(1,1)}$ for $Y$ and, provided that $Y^\super{1,1}=y_1$, we let $R^\super1 = I^\super1_{<y_1}\cap M^\super1$ denote the index set appearing in  Lemma \ref{lemma:decomposition} of Poissons where we still (potentially) condition. 
 
 We want to analyse the term $\bP^\super1_{R^\super1}\left( B_{R^\super1} \;\middle|\;\cC^{\super1, y_1}\right)$  further. 
 For this, let 
$O^\super2=\{i\in O^\super1: y_1(i)=0 \}$ 
be the index set of still not satisfied constraints and set  
\begin{equation}
C_i^{\super1,y_1}:=C^\super1_i\cap I^\super1_{<y_1} \quad \text{ and } \quad \cC_i^{\super1,y_1}=\left\{\sum_{x\in C^{\super1,y_1}_i}N^\super1_x>0\right\}, i\in O^\super2,
\end{equation} 
so that the corresponding event that we still condition on is given by $\cC^{\super1,y_1}=\bigcap_{i\in O^\super2}\cC_i^{\super1,y_1}$. 

Recall that, if $O^\super2$ is empty, then there is no further conditioning and 
$\bP^\super1_{R^\super1} (\cdot \;|\; \cC^{\super1, y_1})$ equals $\bP^\super1_{R^\super1}$,
meaning the remaining $N^\super1_x$, $x\in R^\super1$, are just independent Poissons with rates $\lambda^\super1_x$.

Now assume $O^\super2$ is non-empty, which means $\min(O^\super2) \leq n$ is the first coordinate in $I$ still having an unsatisfied constraint. Let $I^\super{2}:=\{0,1\}^{\{\min(O^\super2),...,n\}}$ and again equip this space with the lexicographical order. Let $\pi_2 : R^\super1 \to I^\super2$ be the canonical projection. Note that $\pi_2$ does not preserve the order, which is why the domain is chosen to be exactly the set over which some conditioning is still performed. 

Next, we group multiple of the Poisson random variables $N^\super1_x$ in $R^\super1$ according to the image of $x$ under $\pi_2$. More precisely, we define $N^\super2_x$, $x\in I^\super2$, via $N^\super2_x=\sum_{z\in \pi_2^{-1}(x)}N^\super1_z$ (with an empty sum being 0), and let $\bP^\super2$ denote their joint law. Note that the sum of independent Poisson random variables is again Poisson, and hence by construction this is a product of independent Poissons with parameters $\lambda^\super2_x = \sum_{z\in \pi_2^{-1}(x)} \lambda^\super1_z$ for $x \in I^\super2$. Let us stress the importance of the term $\pi_2^{-1}(z)$ here, which keeps track of the constituent parts of $N^\super2_z$. 

Lastly, consider on the space $I^\super2$  the constraints, for $ i\in O^\super2$, 
 \begin{align}
 &C^\super2_i = \pi_2(C^{\super1,y_1}_i) \quad \left( = \{x\in I^\super2: x(i)=1\}  \right);
\\&\cC^\super2_i = \{\sum_{x\in C^\super2_i}N^\super2_x>0 \} \quad \left( = \{ \sum_{x\in C^{\super1}_i\cap R^\super1}N^\super1_x>0 \} \right);
\\ &\cC^\super2 = \bigcap_{i\in O^\super2} \cC^\super2_i.
\end{align}
By construction and the above discussion, we thus conclude that, in law, 
\begin{equation}\label{lemma:iteration}
\bP^\super1_{R^\super1}( (N^\super2_x)_{x\in I^\super2} \in \cdot \;|\; \cC^{\super1,y_1}) = \bP^\super2_{I^\super2}(\cdot \;|\; \cC^\super2). 
\end{equation}

What we have done now is a coarsening of the space that prepares us for iterating Lemma \ref{lemma:decomposition} in order to obtain a further decomposition of $\bP_I(\cdot \;|\; \cC)$.

Before we do that, we mention that we can undo this coarsening. In fact, we can re-obtain $\bP^\super1_{R^\super1}$ from $\bP^\super2_{I^\super2}$ via multinomial splitting, that is
\begin{align}
(N^\super1_x)_{x\in \pi_2^{-1}(z)} = \Mult_z(N^\super2_z):=\Mult\left(N^\super2_z; (\frac{\lambda^\super1_x}{Z_z^\super1})_{x\in \pi_2^{-1}(z)}\right)
\end{align}
with $Z^\super1_z = \sum_{x\in \pi_2^{-1}(z)} \lambda^\super1_x$ and where the multinomial sampling is independent over all $z$ and independent of $(N^\super2_z)_{z\in I^\super2}$.

\subsection{The iteration procedure}
We will now perform the general iteration which was described in the previous subsection for one step, introducing the superscript-notation $^\super k$ for the objects at the $k$'th iteration step. Particularly, assume $(O^\super k, \lambda^\super{k}_{I^\super{k}})$ is given where  $O^\super k\subset \{1,...,n\}$ is non-empty and $I^\super k= \{0,1\}^{\{\min(O^\super k),\dots,n\}}$. Let $\bP^\super{k}_{I^\super k}$ be the law of  the independent Poisson random variables $N^{\super k}_x$ with rates $\lambda_x^\super k$, $x\in I^\super k$. For  $i\in O^\super k$, define the constraint sets $C_i^{\super k}=\{x\in I^\super k:x(i)=1\} $, $\cC_i^{\super k} = \{\sum_{x\in C_i^\super k} N_x^\super k > 0\}$, and let $\cC^\super k = \bigcap_{i\in O^\super k} \cC_i^\super k$. 
We set $M^\super{k} = \bigcup_{i\in O^\super k} C^\super{k}_i$ and define 
\begin{equation}Y^{(k,k)} = \max\{x \in M^\super{k} : N_x^\super{k}>0\}.
\end{equation}
Assuming that  $Y^{(k,k)}=y_k \in C^\super{k}_{\min(O^\super k)}$, let  $R^\super{k}=I^\super{k}_{<y_k}\cap M^\super{k}$. 

 If $O^\super{k+1}:= \{i \in O^\super{k} : y_k(i) = 0 \} = \emptyset$, then there is no next level. Otherwise, if $O^\super{k+1}\neq \emptyset$, the next level $(O^\super{k+1},\lambda^\super{k+1}_{I^\super{k+1}})$ is determined by the rates
\begin{align}
\lambda^{\super{k+1}}_x := \sum_{z \in \pi_{k+1}^{-1}(x)} \lambda^\super{k}_z, \quad x\in I^\super{k+1}= \{0,1\}^{\{\min(O^\super{k+1}),...,n\}},
\end{align}
with $\pi_{k+1}: R^\super{k} \to I^\super{k+1}$ again the canonical projection. In that case, consider on the space $I^\super{k+1}$  the constraints, for $ i\in O^\super{k+1}$, 
 \begin{align}
 C^\super{k+1}_i = \pi_{k+1}(C^{\super1,y_k}_i), \: 
\cC^\super{k+1}_i = \{\sum_{x\in C^\super{k+1}_i}N^\super{k+1}_x>0 \},  \:  
\cC^\super{k+1} = \bigcap_{i\in O^\super{k+1}} \cC^\super{k+1}_i.
\end{align}
Then, by a direct application of Lemma \ref{lemma:decomposition}, inserting  the identity \eqref{lemma:iteration} at level $k$, we conclude the following.

\begin{cor}\label{lemma:decompositionK}
Let $B^\super k \subset \bN_0^{I^\super k}$ be a cylinder event. Then we have that 
\begin{align}
&\bP^\super{k}_{I^\super k}(B^\super k \;|\; \cC^\super k) 	\\
&= \prod_{x\in (M^\super k)^c} \bP^\super{k}_x(B^\super{k}_x) \sum_{y_k \in C^\super{k}_{\min(O^\super k)}} \bP^\super{k}_{I^\super{k}}(Y^\super{k,k}=y_k \;|\; \cC^\super{k}) \prod_{x\in I^\super{k}_{>y_k}\cap M^\super{k}} \ind_{0\in B^\super{k}_x} \\
&\quad\cdot \bP^\super{k}_{y_k}(N^\super{k}_{y_k} \in B^\super{k}_{y_k} \;|\; N^\super{k}_{y_k} \geq 1)	
 \cdot \bP^\super{k+1}_{I^\super{k+1}}( ( \text{Mult}_x(N_x^\super{k+1}))_{x \in I^\super{k+1}} \in B_{\pi_{k+1}^{-1} (I^\super{k+1})}\;|\; \cC^\super{k+1}). 
\end{align}
\end{cor}

With Corollary \ref{lemma:decompositionK} at hand, we can thus iterate the procedure by decomposing the term $\bP^\super{k+1}_{I^\super{k+1}}(\cdot \;|\; \cC^\super{k+1})$ further, using Lemma \ref{lemma:decomposition}. 
We next provide the details for undoing this coarsening. For this, for a cylinder set $B^\super k \subset \bN_0^{I^\super k}$, we write 
\begin{align}
\Psi^\super{k}\left( B^\super{k}_{R^\super{k}} \right)
= \int \!\!\!\!\prod_{x\in I^\super{k+1}}\!\!\!\! \bP\left(\Mult(\mathbf{n}_x ; \Lambda^\super{k}_x) \in B^\super{k}_{\pi_{k+1}^{-1}(x)} \right) \ \bP^\super{k+1}_{I^\super{k+1}}(d\mathbf{n}\;|\;\cC^\super{k+1})
\end{align}
to denote the independent multinomial splitting of the conditional Poisson random variables of the level $k+1$, where 
 $\Lambda^\super{k}_x = (\lambda^\super{k}_{z} / \lambda^\super{k+1}_x)_{z\in \pi_{k+1}^{-1}(x)}$ and $\bP^\super{k+1}_{I^\super{k+1}}(\; \cdot \;|\;\cC^\super{k+1})$ is the conditional law determined by the $y_k$-dependent parameters $(O^\super{k+1}, \lambda^\super{k+1}_{I^\super{k+1}})$.

\begin{lem}\label{lemma:iteration2} 
 Let $B^\super k \subset \bN_0^{I^\super k}$ be a cylinder event. Then we have that 
\begin{align}
&\bP^\super{k}_{I^\super k}(B^\super k \;|\; \cC^\super k) 	\\
&= \prod_{x\in (M^\super k)^c} \bP^\super{k}_x(B^\super{k}_x) \sum_{y_k \in C^\super{k}_{\min(O^\super k)}} \bP^\super{k}_{I^\super{k}}(Y^\super{k,k}=y_k \;|\; \cC^\super{k}) \prod_{x\in I^\super{k}_{>y_k}\cap M^\super{k}} \ind_{0\in B^\super{k}_x} \\
&\quad\cdot \bP^\super{k}_{y_k}(N^\super{k}_{y_k} \in B^\super{k}_{y_k} \;|\; N^\super{k}_{y_k} \geq 1)		\\
&\quad \cdot\Big[ \ind_{O^\super{k+1}= \emptyset} \prod_{x\in R^\super k} \bP^\super{k}_x(B^\super{k}_x)   
+\ind_{O^\super{k+1} \not = \emptyset} \Psi^\super{k}(B_{R^\super k}^\super{k}) \Big].
\end{align}
\end{lem}

\begin{proof}
By Lemma \ref{lemma:decomposition}  and Corollary \ref{lemma:decompositionK} we have
\begin{align}
&\bP^\super{k}_{I^\super k}(B^\super k \;|\; \cC^\super k) 	\\
&= \prod_{x\in (M^\super k)^c} \bP^\super{k}_x(B^\super{k}_x) \sum_{y_k \in M^\super{k}} \bP^\super{k}_{I^\super{k}}(Y^\super{k,k}=y_k \;|\; \cC^\super{k}) \prod_{x\in I^\super{k}_{>y_k}\cap M^\super{k}} \ind_{0\in B^\super{k}_x} \\
&\quad\cdot \bP^\super{k}_{y_k}(N^\super{k}_{y_k} \in B^\super{k}_{y_k} \;|\; N^\super{k}_{y_k} \geq 1)	
\bP^\super{k}_{R^\super{k}}(B_{R^\super k}^\super{k}\;|\;\cC^{\super{k}, y_k}) .
\end{align}
By the lexicographical ordering we have that the elements of $C^\super{k}_{\min(O^\super k)}$ are larger than all other elements of $M^\super k$, hence $y_k \in C^\super{k}_{\min(O^\super k)}$.

What remains is the term $\bP^\super{k}_{R^\super{k}}(B_{R^\super k}^\super{k}\;|\;\cC^{\super{k}, y_k})$.  
If $O^\super{k+1} = \emptyset$, then the constraint is empty and the probability is, by independence, simply the product over the one-dimensional Poisson random variables $(N_x)_{x \in R^\super{k}}$. Otherwise, we make the same observation as in the previous subsection that the constraints $\cC_i^{\super{k}, y_k}$ are given by
\[
\cC_i^{\super{k}, y_k}= \left\{ \sum\nolimits_{x\in C^{\super k}_i\cap R^\super k}N^\super{k}_x>0 \right\} = \left\{\sum\nolimits_{x\in C^\super{k+1}_i} \sum\nolimits_{z\in \pi_k^{-1}(x)} N^\super{k}_z > 0 \right\}.
\]
By grouping $\sum\nolimits_{z\in \pi_k^{-1}(x)} N^\super{k}_z$ into new Poisson random variables $N^\super{k+1}_x$ with rates $\lambda^\super{k+1}_x$ we see that the condition $\cC^{\super{k}, y_k}$ is in fact a condition on the values of $N^\super{k+1}_x$ only. The law of the $(N^\super{k}_z)_{z\in R^\super k}$ given $(N^\super{k+1}_x)_{x\in I^\super{k+1}}$ is given by independent multinomial splitting, that is
\[ \bP^\super{k}\left( B^\super{k}_{R^\super{k}} \right)
= \int \!\!\!\!\prod_{x\in I^\super{k+1}}\!\!\!\! \bP\left(\Mult(\mathbf{n}_x ; \Lambda^\super{k}_x) \in B^\super{k}_{\pi_{k+1}^{-1}(x)} \right) \ \bP^\super{k+1}_{I^\super{k+1}}(d\mathbf{n}). \]
Conditioning on $\cC^{\super{k}, y_k}$ then gives us $\Psi^\super{k}(B^\super{k}_{R^\super k})$. Particularly, we conclude that
\begin{equation}
\bP^\super{k}_{R^\super{k}}(B_{R^\super k}^\super{k}\;|\;\cC^{\super{k}, y_k}) = \Big[ \ind_{O^\super{k+1}= \emptyset} \prod_{x\in R^\super k} \bP^\super{k}_x(B^\super{k}_x)   
+\ind_{O^\super{k+1} \not = \emptyset} \Psi^\super{k}(B_{R^\super k}^\super{k}) \Big]
\end{equation}
and therefore, by the above discussion, also the proof.
\end{proof}

\subsection{The full decomposition of $\bP_I(\cdot \;|\; \cC)$.}

Here we combine the theory from the previous subsections to obtain a more detailed decomposition of $\bP(\cdot | \cC)$. As an application thereof we give the proof of Theorem \ref{thm:domination2}, summarized in the next subsection.

The starting point is Lemma \ref{lemma:decomposition}. By utilising  Corollary  \ref{lemma:decompositionK} and Lemma \ref{lemma:iteration2}, we can iterate this decomposition. 
Since $|O|<\infty$, this procedure will stop at the random time $\kappa = \inf\{ k \geq 1 \colon O^\super{k} = \emptyset \} \leq n$. This provide us with elements $Y^\super{i,i}=y_i$ for $i =1,\dots, \kappa$, where $Y^\super{i,i} \in I^\super{i}$ represents the explored points that combined guarantee that the constraints are satisfied. 

There are two things that we still need to handle. Firstly,  the points $Y^\super{i,i}$ for $i=2,\dots, \kappa$ must be lifted to the space $I^\super1$, as must all other points $x \in I^\super{k}$ for $k>1$ contained in the decomposition of $\bP_I(\cdot \;|\; \cC)$ (i.e.\ those corresponding to $x \in ( M^\super{i} )^c$ for some $i \in \{1,\dots, \kappa \}$,  that their are no points corresponding to $x\in I^\super{k}_{>y_i}\cap M^\super{i}$ and lastly, possibly additional copies of the $y_i$'s). We explain in the following paragraph how this can be done using multinomial sampling. Secondly,  the result of this multinomial sampling procedure  will depend on the constraints and the points $y_1,\dots,y_{\kappa}$ in a seemingly highly not-trivial way which appears to be challenging to control explicitly. To overcome this complexity, we apply the domination property of Lemma \ref{lem:basicPoisson} as in  the proof of  Lemma \ref{lem:firstDomiGeneral} at each iteration step $k=1,\dots, \kappa$ before lifting the points to the space $I^\super1$.  As a consequence of this and the multinomial sampling, since no points besides the $Y^\super{i,i}$'s are counted twice in the decomposition, all points at the level of $I^\super1$ are stochastically dominated by points sampled according to the unconditional law $\bP_I(\cdot)$ and independently the additional lifts of $Y^\super{i,i}$, $i=1,\dots,\kappa$, of which there are $n=n(\lambda)$ many.

We now describe how to perform the first step. Given $y_k\in I^\super{k}$, we want to lift this to a random variable $Y^\super{k,1} \in I^\super1$ with $\opi_k(Y^\super{k,1})=y_k$, where $\opi_k = \pi_k \circ \pi_{k-1} \circ \cdots \circ \pi_2$. It is important to note here that $\opi_k$ is not defined on the entirety of $I^\super1$, but only on the subset $\opi_k^{-1}(I^\super{k}) := \pi_2^{-1} \circ \cdots \circ \pi_k^{-1}(I^\super{k})$.  We observe that $\lambda_{y_k}^\super{k} = \sum_{z \in \opi_k^{-1}(y_k)} \lambda_z^\super1$. The random lift $Y^\super{k,1}$ of $y_k$ is then defined as a random variable $Y^\super{k,1} \in \opi_k^{-1}(y_k)$ sampled independently from everything else (given $y_1,...,y_k$) with the probabilities $\lambda^\super1_z / \lambda^\super{k}_{y_k}$, $z\in \opi_k^{-1}(y_k)$. 
It follows from the domains of $\pi_\ell$ that $Y^\super{k,1}$ satisfies $\opi_\ell(Y^\super{k,1}) < y_\ell$ for all $\ell<k$ and $\opi_k(Y^\super{k,1})=y_k$. We remark here that lifting $y_k$ to $Y^\super{k,k-1} \in \pi_k^{-1}(y_k) \subset I^\super{k-1}$ with probabilities $\lambda_z^\super{k-1}/\lambda_{y_k}^\super{k}$, then lifting $Y^\super{k,k-1}$ to $Y^\super{k,k-2}\in \pi_{k-1}^{-1}(Y_k^\super{k-1}) \subset I^\super{k-2}$ with probabilities $\lambda_z^\super{k-2}/\lambda_{Y_k^\super{k-1}}^\super{k-1}$ and so on will give the same law of $Y^\super{k,1}$.

The above lifting procedure can also be applied to the other points  $x \in I^\super{k}$ for $k>1$ contained in the decomposition of $\bP_I(\cdot \;|\; \cC)$. 
For each  $x \in ( M^\super{i} )^c$ with $i \in \{1,\dots, \kappa \}$ we sample $N_x^\super{i} \sim Poi(\lambda_x^\super{i})$ and there next, independently, apply the same multinomial sampling procedure as described above for $N_x$ copies of $x$. The point $x\in I^\super{k}_{>y_i}\cap M^\super{i}$ are simply not sampled. Lastly, we sample $N_{y_i}$ according to the law $\bP_{y_i}^\super{i} ( N_{y_i} \in | N_{y_i}\geq 1)$ and insert an additional $N_{y_i}-1$ copies of $y_i$ to which, again independently, we apply the same multinomial sampling procedure.

  Summarizing the above, and arguing as in the proof of  Lemma \ref{lem:firstDomiGeneral},  we end this subsection with the following statement.

\begin{thm}\label{thm:domination2GeneralOlD}
The probability measure $\bP^\super{1}_{I^\super 1}( \cdot \;|\; \cC^\super 1)$ is equal, in law, to  
\begin{equation}\sum_{\bfy} \bP^\super{1}_{I^\super{1}}(\bfy \;|\; \cC^\super{1}) \sum_{k=1}^{\kappa}\cL(Y^\super{k}|\bfy), \quad \text{ where:}\end{equation}
\begin{itemize}
    \item $\bfy = (y_1,...,y_K)$ with $K\leq \kappa$, and $\bP^\super{1}_{I^\super{1}}(\bfy \;|\; \cC^\super{1})$ is the probability that the iteration of Lemma \ref{lemma:decomposition} using Corollary \ref{lemma:decompositionK} terminates after $K$ steps and $Y^\super{k,k}=y_k$, $k=1,...,K$;
    \item For $k=1,\dots,K$, $\cL(Y^\super{k}|\bfy) =\cL(Y^\super{1,k}|\bfy)$ is the law of the random lifts $Y^\super1_k \in I^\super1$ of $Y^\super{k,k} = y_k$ given $\bfy = (y_1,...,y_K)$.
    \item For $k>K$, $\cL(Y^\super{k}|\bfy)$ is the law of the additional random lifts of points appearing from the iteration procedure.
    \item Given $\bfy$, the random variables $(Y^\super{k}|\bfy)$ are independent.
\end{itemize}
Furthermore, there is an $n \in \bN$ such that $\bP^\super{1}_{I^\super 1}( \cdot \;|\; \cC^\super 1)$ is stochastically dominated by 
\begin{align}
\bP^\super{1}_{I^\super{1}}  \bigoplus \sum_{\bfy} \bP^\super{1}_{I^\super{1}}(\bfy \;|\; \cC^\super{1}) \sum_{k=1}^{|\bfy|}\sum_{j=1}^{n}\cL(Y_j^\super{1,k}|\bfy),
\end{align}
where, given $\bfy = (y_1,...,y_K)$, the $Y_j^\super{1,k}$ are independent random lifts  of $Y^\super{k,k} = y_k$.
\end{thm}

\begin{proof}
The first decomposition and the following four properties of the random lifts $Y^\super1_k \in I^\super1$  follow from the discussion above. To see that the stochastic domination statement holds, we note that the argument provided in the proof Lemma \ref{lem:firstDomiGeneral} still applies. Indeed, at each iteration step, using Corollary  \ref{lemma:decompositionK}, we have that
$\prod_{x\in I^\super{k}_{>y_k}\cap M^\super{k}} \ind_{0\in \cdot }$ is trivially dominated by corresponding independent Poisson random variables. 
Similarly, utilising Lemma \ref{lem:basicPoisson}, the terms $\bP^\super{k}_{y_k}(N^\super{k}_{y_k} \in \cdot \;|\; N^\super{k}_{y_k} \geq 1)$ are stochastically dominated by $n+N^\super{k}_{y_k}$, where the $N^\super{k}_{y_k}$'s are independent $Poi(\lambda_{y_k}^\super{k})$ random variables. 
\end{proof}

\subsection{Proof of Theorem \ref{thm:domination2}}

In this last subsection, we turn back to the original motivation and the field of random walks-model and, equipped with the theory from the previous subsections, argue that  Theorem \ref{thm:domination2} indeed holds. 

As explained in Subsection \ref{sec:SDfRW}, since conditioning on $A$ does not affect $\xi^{\gamma^c}$, we can split that part of the field of and focus on the effect of the condition on $\xi^\gamma$. Further, by projecting the latter onto $I^\super{1}$ we can control the effect of conditioning on $A$ via the scheme outlined in the previous subsections. Particularly, writing $\bP^\super{1}_{I^\super 1}( \cdot \;|\; \cC^\super 1)$ for this projected law, by Theorem \ref{thm:domination2GeneralOlD} we have that 
\begin{equation}
\bP^\super{1}_{I^\super 1}( \cdot \;|\; \cC^\super 1)=\sum_{\bfy} \bP^\super{1}_{I^\super{1}}(\bfy \;|\; \cC^\super{1}) \sum_{k=1}^{\kappa}\cL(Y^\super{k}|\bfy).
\end{equation}
Moreover, at the level of the  $Y^\super{k}$'s, the independence and the stochastic domination properties described there holds too. The following lemma is used in order to lift the elements $Y^\super{k}$ to random walk paths. Before this, we first note that, since this will be done independently for each $Y^\super{k}$, the independence and the stochastic domination property of Theorem \ref{thm:domination2GeneralOlD} immediately transfers to these elements too, yielding that Properties \eqref{enum:independence} and \eqref{enum:domination} of Theorem \ref{thm:domination2} hold.

Now, given $\bfy$, for $k=1,\dots, |\bfy|$, let $D_k= D_k(\bfy)$ be the event that a random walk path $X$ do not intersect $\gamma$ for times in $V$ and obey the restrictions provided by the lifting mechanism of the corresponding element $y_k$. That is,  
\begin{equation}
D_k := \{\opi_k(\trace_{\gamma} (X))=y_k, \opi_{\ell}(\trace_\gamma(X)) <y_{\ell}, \ell=1,...,k-1, X_{-s}\neq \gamma_s \ \forall\;s\in V\}.
\end{equation}

\begin{lem}\label{lem:LastOne}
Let $Y_k^\super{1} \in I^\super1$ be the random lift of $y_k \in I^\super k$. Then 
\begin{align}
\bP(Y_k^\super{1}=x) = P^{RW}(X \in x \;|\;  D_k), \quad x \in  I^\super1.
\end{align}
\end{lem}

\begin{proof}
Firstly, we note that, for any $x\in I^\super1$ with $x(i)=0 \ \forall\; i\in V$ and $x(k)=1$ we have
\begin{align}\label{eq:lemma:rw-lambda}
\lambda_x^\super1 = \lambda \bP^{RW}(X \in x \;|\; X_{-k}=\gamma_k)
\end{align}
Indeed, it holds that 
\begin{align}
\lambda^\super1_x 
&= \lambda \sum_{z\in\bZ^d} \bP^{RW}(X \in x \;|\; X_0=z)   \\
&= \lambda \sum_{z\in\bZ^d} \bP^{RW}(X \in x \;|\; X_{-k}=\gamma_k, X_0=z ) \bP^{RW}(X_{-k}=\gamma_k \;|\; X_0=z )   \\
&= \lambda \sum_{z\in\bZ^d} \bP^{RW}(X \in x \;|\; X_{-k}=\gamma_k, X_0=z ) \bP^{RW}( X_0=z \;|\; X_{-k}=\gamma_k )   \\
&= \lambda \sum_{z\in\bZ^d} \bP^{RW}(X \in x, X_0=z \;|\; X_{-k}=\gamma_k ) = \lambda \bP^{RW}(X \in x \;|\; X_{-k}=\gamma_k ). 
\end{align}
By definition of the lift, $\bP(Y_k^\super{1}=x) = \lambda_x^\super1 / \lambda^\super{k}_{y_k}$, $x\in \opi_k^{-1}(y_k)$ (and 0 for other $x$). Since $x({\min(O^\super{k})})=1$ for all $x\in \opi_k^{-1}(y_k)$ we can apply \eqref{eq:lemma:rw-lambda} to get
\begin{align}
\bP(Y_k^\super{1}=x) 
&= \frac{\lambda P^{RW}(X \in x \;|\; X_{-\min(O^\super{k})}=\gamma_{\min(O^\super{k})})}{\lambda \sum_{z \in \opi_k^{-1}(y_k)}P^{RW}(X \in z \;|\; X_{-\min(O^\super{k})}=\gamma_{\min(O^\super{k})})\ind_{z(i)=0\;\forall i\in V}}     \\
&= \frac{P^{RW}(X \in x \;|\; X_{-\min(O^\super{k})}=\gamma_{\min(O^\super{k})})}{P^{RW}( X \in \opi_k^{-1}(y_k), X_{-s}\neq \gamma_s \ \forall\;s\in V \;|\; X_{-\min(O^\super{k})}=\gamma_{\min(O^\super{k})})}  \\
&= P^{RW}(X \in x \;|\; X \in \opi_k^{-1}(y_k), X_{-s}\neq \gamma_s \ \forall\;s\in V, X_{-\min(O^\super{k})}=\gamma_{\min(O^\super{k})}),
\end{align}
where we have used that the event $\{X \in x\}$ with $x \in \opi_k^{-1}(y_k)$ is contained in $\{X \in \opi_k^{-1}(y_k), X_{-s}\neq \gamma_s \ \forall\;s\in V\}$.
Particularly, we have that 
\begin{align}
\opi_k^{-1}(y_k) = \left\{ x\in I^\super1 : \opi_\ell(x) \in I^\super{\ell}_{<y_{\ell}} \cap M^\super\ell, 1\leq \ell \leq k-1, \opi_k(x)=y_k \right\},
\end{align}
where the domains $R^\super{\ell} = \pi_{\ell+1}^{-1}(I^\super\ell) = I^\super{\ell}_{<y_{\ell}} \cap M^\super\ell$ enter because of the definition of $\opi_k$. Since $y_k$ satisfies at least the constraint with index $\min(O^\super{k})$ (since $y_k(\min(O^\super{k}))=1$), and level $k$ is the first level to do so, this in turn implies that $\pi_{\ell+1}^{-1}\circ...\circ \pi_k^{-1}(y_k) \in C_{\min(O^\super{k})}^{\ell}$ for all $1\leq \ell \leq k-1$. Hence all preimages of $y_k$ on level $\ell$ lie in $M^\super\ell$, $1\leq \ell \leq k-1$, and we can drop the intersection with $M^\super\ell$, which gives us
\begin{align}
\opi_k^{-1}(y_k) = \left\{ x\in I^\super1 : \opi_\ell(x)<y_\ell, 1\leq \ell \leq k-1, \opi_k(x)=y_k \right\},
\end{align}
which completes the proof.
\end{proof}

\begin{proof}[Proof of Theorem \ref{thm:domination2}, Property \eqref{enum:rw}]

Given $\bfy$, independently, we sample $Z^\super{k}$ in $\cS$ as follows. First, we sample $Y^\super{k}$ yielding an element $x \in I^\super{1}$. By Lemma \ref{lem:LastOne}, this agrees with the probabilities $P^{RW}(X \in x \;|\;  D_k)$. Given $x \in I^\super{1}$, we then sample  $Z^\super{k}$ in $\cS$ using the law $P^\super{RW}( \cdot | D_k, \{Z^\super{k} \in  x\})$. By the Markov property of random walk paths, the law of $Z^\super{k}$ is thus exactly $P^\super{RW}(\cdot | D_k)$.

By this construction, $D_k$ only concerns the behaviour of $Z^\super{k}$ in the time interval $[-|\gamma|,-1]$. Particular, given $Z^\super{k}_{-1}= z \in \bZ^d$, it evolves as a simple random walks for all future times. Moreover, since the conditioning on $D_k$ guarantees that $Z^\super{k}_{-\min(O^\super{k})} = \gamma_{-\min(O^\super{k}}$, letting $z_k= -\min(O^\super{k})$ provides us with the corresponding anchor point.

 Lastly,  the claim that  $D_i$ satisfies $\{Z^\super{i}_{z_i}=\gamma_{z_i}\}\cap\{ Z^\super{i}_s \neq \gamma_s, z_i<s\leq -1 \} \subset D_i \subset \{Z^\super{i}_{z_i}=\gamma_{z_i}\}$ and is a deterministic function of $\bfy$ is then an immediate consequence of its definition. 
\end{proof}

\begin{proof}[Proof of Theorem \ref{thm:domination2}, Property \eqref{enum:bounds}]
Firstly, we note that $D_k$ contains the event that  $Z^\super{k}$ hits $\gamma$ only at the time point provided by the anchor point $z_k$. Assuming that $d\geq 3$, it thus follows by Lemma \ref{lem:SRWestimates}
that there is $\delta>0$ not depending on $\gamma$ nor $A$ such that $P^{RW}(D_k)>\delta$. Moreover, this implies that 
\begin{align}
\bQ_A(Z^{i}_t = x \;|\; \bfy) = & P^\super{RW} (  Z^{i}_t = x \;|\; D_k) 
\\ = &P^\super{RW} (  Z^\super{i}_t = x , Z^\super{i}_t =x, Z^\super{i}_{z_i} =\gamma_{z_i},  \;|\; D_k)
\\  \leq  & \delta^{-1}  P^\super{RW} (  Z^\super{i}_t = x , Z^\super{i}_t =x, Z^\super{i}_{z_i} =\gamma_{z_i})
\\ \leq & \delta^{-1} C(t-z_i)^{-\frac d2 }.
\end{align}
where we in the last inequality again apply Lemma \ref{lem:SRWestimates}, and from which we conclude that 
 \eqref{enum:bound} holds.
\end{proof}

\end{document}